\documentclass[12pt]{article}
\usepackage[a4paper,margin=1.1in]{geometry}
\usepackage{amssymb}
\usepackage{amsmath}
\usepackage{amsfonts}
\usepackage{amsthm}
\usepackage{color}
\usepackage{float}
\usepackage[all]{xy}
\usepackage{delimset}

\def\ward{\mathop{\mbox{\textsl{W}}}\nolimits}
\def\Exp{\mathop{\mbox{\textup{Exp}}}\nolimits}
\def\E{\mathop{\mbox{\textup{E}}}\nolimits}

\newcommand{\fibonomial}{\genfrac{\{}{\}}{0pt}{}}

\newcommand{\C}{\mathbb{C}}
\newcommand{\Z}{\mathbb{Z}}
\newcommand{\K}{\mathbb{K}}
\newcommand{\N}{\mathbb{N}}

\newcommand{\R}{\mathbb{R}}
\newcommand{\Per}{\mathbb{P}}
\newcommand{\e}{\mathrm{e}}
\newcommand{\EE}{\mathrm{E}}

\newtheorem{theorem}{Theorem}

\newtheorem{definition}{Definition}
\newtheorem{proposition}{Proposition}
\newtheorem{corollary}{Corollary}
\newtheorem{example}{Example}

\begin{document}

\title{\textbf{Existence of Solutions of Functional-Difference Equations with Proportional Delay on Deformed Generalized Fibonacci Polynomials via Successive Approximation and Bell Polynomials}}
\author{Ronald Orozco L\'opez}
\newcommand{\Addresses}{{
  \bigskip
  \footnotesize

  \textit{E-mail address}, R.~Orozco: \texttt{rj.orozco@uniandes.edu.co}
  
}}

\maketitle
\tableofcontents

\begin{abstract}
In this paper, we study the existence of solutions of the functional difference equations with proportional delay on deformed generalized Fibonacci polynomials via successive approximation method and Bell polynomials. First, we introduce the deformed generalized Fibonacci polynomials and show that the $q$-numbers can be viewed as "bifurcation" of deformed $(s,t)$-numbers. These deformations are closely related to proportional delay. Second, a differential and integral calculus on deformed generalized Fibonacci polynomials is introduced. The main reason for introducing this calculation is to have a framework for solving proportional functional equations and thus obtain the Pell calculus, Jacobsthal calculus, Chebysheff calculus, and Mersenne calculus, among others. We study the convergence of $(s,t)$-exponential type series and its dependence on the deformation parameter. We define the deformed $(s,t)$-exponential functions and we give its analytic and algebraic properties. In addition, we study the $(1,u)$-deformed $(s,t)$-exponential function and use it to prove the existence of functional difference equations with proportional delay. The solution is not unique when it is related to $q$-periodic functions.
\end{abstract}
\noindent 2020 {\it Mathematics Subject Classification}:
Primary 39B22 Secondary 11B39, 11B65, 30B10, 34K06.

\noindent \emph{Keywords: } Fibonacci polynomial, deformed functions, Pell calculus, Jacobsthal calculus, Chebysheff calculus, Mersenne calculus, functional-difference equation, Pantograph equation.

\section{Introduction}

The main objective of this paper is to study the existence of solutions of functional equations of the form
\begin{equation}\label{eqn_funct}
    y(ax)-y(bx)=(a-b)x f(x,y(x),y(cx)),\ y(\eta)=\xi,
\end{equation}
$a,b,c\in\R$, or equivalently, the existence of solutions of functional difference equation with proportional delay
\begin{equation}\label{eqn_funct2}
    \frac{y(ax)-y(bx)}{(a-b)x}=f(x,y(x),y(ux)),\ y(\eta)=\xi.
\end{equation}
The difference quotient of the left-hand side of Eq.(\ref{eqn_funct2}) reminds us of the derivative operator of the $(p,q)$-calculus
\begin{equation}\label{eqn_pq_diff}
    \mathbf{D}_{p,q}f(x)=
    \begin{cases}
    \frac{f\left(px\right)-f\left(qx\right)}{(p-q)x},&\text{ if } x\neq0;\\
    f^{\prime}(0), &\text{ if } x=0,
    \end{cases}
\end{equation}
provided that $f$ is differentiable at 0. Therefore, from Eq.(\ref{eqn_pq_diff}) follows that Eq.(\ref{eqn_funct2}) is equivalent to
\begin{equation}\label{eqn_funct3}
    \mathbf{D}_{a,b}y(x)=f(x,y(x),y(u x)).
\end{equation}
Then to construct a theory to solve problems involving functional equations of the form Eq.(\ref{eqn_funct}) it is necessary to use the $(p,q)$-calculus, but since $a$ and $b$ are fixed real numbers, we wish to use a more specialized calculus. For example, if we want to solve the equation
\begin{equation*}
    f(2x)-f(-x)=3xf(x)+3xf(ux),\ f(0)=1,\ 0<u<2,
\end{equation*}
we must construct a calculus on the Jacosbthal numbers
\begin{equation*}
      J_n=(0,1,1,2,3,5,11,21,43,85,171,\ldots).
\end{equation*}
Then we must construct a new calculus, and in a very optimistic way, a new mathematics, on sequences of numbers. Fonten\'e in \cite{font} published a paper in which he generalized the binomial coefficients by replacing $\binom{n}{k}=\frac{n(n-1)\cdots(n-k+1)}{1\cdot2\cdots k}$, consisting of natural numbers, with $\binom{n}{k}_\psi=\frac{\psi_n\psi_{n-1}\cdots\psi_{n-k+1}}{\psi_1\psi_2\cdots\psi_k}$, formed by an arbitrary sequence $\psi=\{\psi_n\}$ of real or complex numbers. He gave a fundamental recurrence relation for these coefficients such that when we make $\psi_n=n$ we recover the ordinary binomial coefficients and when we make $\psi_n=\brk[s]{n}_q=\frac{q^{n}-1}{q-1}$ we recover the $q$-binomial coefficients studied by Euler \cite{euler}, Gauss \cite{gauss_1,gauss_2}, Jackson \cite{jackson,jackson_2,jackson_3,jackson_4,jackson_5,jackson_6,jackson_7} and others \cite{kac}.

Subsequently, Ward \cite{ward} developed a symbolic calculus on sequences $\psi=\{\psi_n\}$ with $\psi_0=0$, $\psi_1=1$, and $\psi_n\neq0$ for all $n\geq1$, and thus generalized the ordinary calculus and the $q$-calculus of Jackson. Recently, Chakrabarti and Jagannathan \cite{jag}, Brodimas et al. \cite{brodi}, Wachs and White \cite{wahcs}, and Arik et al. \cite{arik} developed a new calculus, extension of the $q$-calculus, called the $(p,q)$-calculus or Post Quantum Calculus. Other well-studied calculus emerged from his work, the Fibonomial calculus, where $\psi_n=F_n$ is the Fibonacci sequence defined recursively by $F_0=0$, $F_1=1$, $F_{n+1}=F_n+F_{n-1}$. For more details on some works on this subject see \cite{pashaev_1,pashaev_2,ozvatan}. 

In this paper, we investigate a Ward calculus defined on generalized Fibonacci polynomials
\begin{align}\label{eqn_fibo}
 F_{n+1}&=sF_n+tF_{n-1},
\end{align}
with initial values $F_0=0$ and $F_1=1$, in the variables $s,t$. As special cases, we obtain the integral and differential calculus of Pell, Jacobsthal, Chebysheff of the second kind, Mersenne, and Repunits, among others. When $s=P$ and $t=-Q$, with $P$ and $Q$ integers, we obtain the $(P,-Q)$-Lucas differential calculus and if $Q=-1$, we obtain the $P$-Fibonacci differential calculus. 

On the other hand, we have the following definition of $(p,q)$-number
\begin{equation*}
    \brk[s]{n}_{p,q}=\frac{p^n-q^n}{p-q}.
\end{equation*}
When $p=1$, the $(p,q)$-numbers reduce to the $q$-numbers $\brk[s]{n}_q$. In general, the $(p,q)$-calculus reduce to the $q$-calculus when $p=1$. The $(p,q)$-analogue of $n!$ is
\begin{equation*}
    \brk[s]{n}_{p,q}!=
    \begin{cases}
        \brk[s]{1}_{p,q}\brk[s]{2}_{p,q}\cdots\brk[s]{n-1}_{p,q}\brk[s]{n}_{p,q}, &\text{ if }n\geq1;\\
        1,&\text{ if }n=0.
    \end{cases}
\end{equation*}
From here the $(p,q)$-binomial coefficients are
\begin{equation*}
    \binom{n}{k}_{p,q}=\frac{\brk[s]{n}_{p,q}}{\brk[s]{k}_{p,q}\brk[s]{n-k}_{p,q}}.
\end{equation*}
Some basic $(p,q)$-functions are: the $(p,q)$-analogue of $(x-a)^n$
\begin{equation*}
    (x\ominus a)_{p,q}^{n}=
    \begin{cases}
        \prod_{k=0}^{n-1}(p^kx-q^ka),&\text{ if }n\geq1;\\
        1,&\text{ if }n=0,
    \end{cases}
\end{equation*}
and the $(p,q)$-exponential functions
\begin{align}
    \e_{p,q}(x)&=\sum_{n=0}^{\infty}p^{\binom{n}{2}}\frac{x^n}{\brk[s]{n}_{p,q}!},\label{eqn_exp_p}\\
    \E_{p,q}(x)&=\sum_{n=0}^{\infty}q^{\binom{n}{2}}\frac{x^n}{\brk[s]{n}_{p,q}!}\label{eqn_exp_q}.
\end{align}
Pashaev et al. \cite{pashaev_1} introduced the Fibonomial calculus or Golden $q$-calculus as a special case of the $q$-calculus. They defined the golden derivative of the function $f(x)$ as
\begin{equation*}
    \mathbf{D}_{F}f(x)=
    \begin{cases}
    \frac{f\left(\varphi x\right)-f\left(-x/\varphi\right)}{(\varphi+\frac{1}{\varphi})x},&\text{ if } x\neq0;\\
    f^{\prime}(0), &\text{ if } x=0.
    \end{cases}
\end{equation*}
Also, the Golden exponential functions are
\begin{align}
    e^x_{F}&=\sum_{n=0}^{\infty}\frac{x^n}{F_n!},\label{eqn_exp_fibo}
    \end{align}
and
\begin{align}
    \E^x_{F}&=\sum_{n=0}^{\infty}(-1)^{\frac{n(n-1)}{2}}\frac{x^n}{F_n!},\label{eqn_exp_fibom}
\end{align}
where $F_n!=F_1F_2\cdots F_n$ is the $F$-analogue of $n!$. The Golden Binomial is
\begin{equation*}
    (x+y)_{F}^{n}=(x+\varphi^{n-1}y)(x-\varphi^{n-3}y)\cdots(x+(-1)^{n-1}\varphi_{-n+1}y).
\end{equation*}
The Fibonomial calculus also turns out to be a special case of the $(p,q)$-calculus. Every calculus obtained from Eq.(\ref{eqn_fibo}) can be seen as a special case of the $(p,q)$-calculus. It is enough to make $p+q=s$ and $pq=-t$ to obtain
\begin{align*}
    p&=\frac{s+\sqrt{s^2+4t}}{2},\\
    q&=\frac{s-\sqrt{s^2+4t}}{2}.
\end{align*}
In this paper, we will restrict ourselves only to the polynomials $\brk[c]{n}_{s,t}$ such that $\brk[c]{n}_{s,t}\in\R[s,t]$, i.e., when $s\neq0$ and $s^2+4t>0$. With this family of generalized Fibonacci polynomials, we obtain the Fibonacci, Pell, Jacobsthal, and Mersenne numbers and the Chebysheff of second-order polynomials. It is well known that these sequences have a wide number of applications in fields as diverse as combinatorics, number theory, algebra, computer science, patterns in nature, optics physics, music, and others \cite{thomas}, \cite{thomas2}, \cite{manson}. Just as Fibonacci calculus also has direct application in quantum mechanics, we hope with this paper to construct new calculus with possible applications in mathematics and physics.

Another example of a functional difference equation is
\begin{equation*}
    \mathbf{D}_{a,b}f(x)=f(cx),\ f(0)=1
\end{equation*}
which is the $(a,b)$-analog of the Pantograph functional differential equation
\begin{equation}\label{eqn_sokal}
    y^{\prime}(x)=y(qx),\ \ f(0)=1.
\end{equation}
Morris et al. \cite{morris}, Langley \cite{lan}, and Mahler \cite{mahler} studied the function
\begin{equation*}
    \Exp(x,q)=\sum_{n=0}^{\infty}q^{\frac{n(n-1)}{2}}\frac{x^n}{n!}
\end{equation*}
which is the solution of Eq.(\ref{eqn_sokal}). The function $\Exp(x,y)$ is a deformed exponential function since when $q\rightarrow 1$, then $\Exp(x,y)\rightarrow e^x$. It is closely related to the generating function for the Tutte polynomials of the complete graph $K_n$ in combinatorics \cite{gessel2}, the partition function of one-site lattice gas with fugacity $x$ and two-particle Boltzmann weight $q$ in statistical mechanics \cite{sokal}, cell division \cite{brunt}, and brightness of the galaxy \cite{ambart}.

As the functions in Eqs.(\ref{eqn_exp_p}), (\ref{eqn_exp_q}), (\ref{eqn_exp_fibo}), and (\ref{eqn_exp_fibom}) satisfy the equations 
\begin{align}
    \mathbf{D}_{p,q}f(x)&=f(ux),\ \ u=p,q,\label{eqn_panto_pq}\\
    \mathbf{D}_{F}f(x)&=f(ux),\ \ u=1,-1\label{eqn_panto_fibo}
\end{align}
with $f(0)=1$, which are analogous to the Pantograph differential equation in Eq. (\ref{eqn_sokal}), we can then call such functions 
deformed exponential functions. Another well-studied Pantograph equation is \cite{ebaid}, \cite{ebaid2}, \cite{kato}, \cite{ock}
\begin{equation}\label{eqn_pantograph}
    y^{\prime}(x)=ay(x)+by(qx).
\end{equation}
We will find the solution of the $(s,t)$-analog of the Eq.(\ref{eqn_pantograph}) and we will use this to prove the existence of solutions of Eq.(\ref{eqn_funct3}). Generalizations of the Eqs.\ref{eqn_sokal}) and (\ref{eqn_pantograph}) can be found in \cite{ise_1,ise_2}.

We divide this paper as follows. In Section 2, we introduce deformed generalized Fibonacci polynomials in the variables $s,t$, with their respective most important specializations. In addition, from the point of view of the theory of difference equations, the deformed generalized Fibonacci polynomials have the q-numbers as bifurcation points. In Section 3, we define the Ward ring of deformed $(s,t)$- exponential generating functions and we show that the convergence of these series depends strongly on the deformation parameter. In addition, differential and integral calculus on generalized Fibonacci polynomials are introduced. Also, we introduce the $q$-periodic functions, which play an important role in the existence of multiple solutions of Eq.(\ref{eqn_funct3}). Finally, in this section, we study the deformed $(s,t)$-exponential functions, and some of their analytical properties are given. In section 4, we study the existence of solutions of Eq.(\ref{eqn_funct3}) via the successive approximation method and Bell polynomials. As a particular case, we study the solutions of the equations $\mathbf{D}_{s,t}y(x)=ay(ux)$ and $\mathbf{D}_{s,t}y(x)=ay(x)+by(ux)$.

\section{Deformed generalized Fibonacci polynomials}

\subsection{Basic properties}

The generalized Fibonacci polynomials depending on the variables $s,t$ are defined by
\begin{align*}
    \brk[c]{0}_{s,t}&=0,\\
    \brk[c]{1}_{s,t}&=1,\\
    \brk[c]{n+2}_{s,t}&=s\{n+1\}_{s,t}+t\{n\}_{s,t}.
\end{align*}
Below are some important specializations of Fibonacci polynomials.
\begin{enumerate}
    \item When $s=0,t=0$, then $\brk[c]{0}_{0,0}=0$, $\brk[c]{1}_{0,0}=1$ and $\brk[c]{n}_{0,0}=0$ for all $n\geq2$.
    \item When $s=0,t\neq0$, then $\brk[c]{2n}_{0,t}=0$ and $\brk[c]{2n+1}_{0,t}=t^n$.
    \item When $s\neq0,t=0$, then $\brk[c]{n}_{s,0}=s^{n-1}$.
    \item When $s=2,t=-1$, then $\brk[c]{n}_{2,-1}=n$, the positive integer.
    \item When $s=1,t=1$, then $\brk[c]{n}_{1,1}=F_n$, the Fibonacci numbers.
    \item When $s=2,t=1$, then $\brk[c]{n}_{2,1}=P_n$, where $P_n$ are the Pell numbers
    \begin{equation*}
        P_n=(0,1,2,5,12,29,\ldots).  
    \end{equation*}
    \item When $s=1,t=2$, then $\brk[c]{n}_{1,2}=J_n$, where $J_n$ are the Jacosbthal numbers
    \begin{equation*}
        J_n=(0,1,1,2,3,5,11,21,43,85,171,\ldots).
    \end{equation*}
    \item When $s=p+q,t=-pq$, then $\brk[c]{n}_{p+q,-pq}=\brk[s]{n}_{p,q}$, where $\brk[s]{n}_{p,q}$ are the $(p,q)$-numbers
    \begin{equation*}
        \brk[s]{n}_{p,q}=(0,1,\brk[s]{2}_{p,q},[3]_{p,q},[4]_{p,q},[5]_{p,q},[6]_{p,q},[7]_{p,q},[8]_{p,q}\ldots).
    \end{equation*}
    \item When $s=2t,t=-1$, then $\brk[c]{n}_{2t,-1}=U_{n-1}(t)$, where $U_n(t)$ are the Chebysheff polynomials of the second kind, with $U_{-1}(t)=0$
    \begin{equation*}
        U_n(t)=(0,1,2t,4t^2-1,8t^3-4t,16t^4-12t^2+1,32t^5-32t^3+6t,\ldots).
    \end{equation*}
    \item When $s=3,t=-2$, then $\brk[c]{n}_{3,-2}=M_n$, where $M_n=2^n-1$ are the Mersenne numbers
    \begin{equation*}
        M_n=(0,1,3,7,15,31,63,127,255,\ldots).
    \end{equation*}
    \item When $s=b+1,t=-b$, then $\brk[c]{n}_{b+1,-b}=R_n^{(b)}$, where $R_n^{(b)}$ are the Repunit numbers in base $b$
    \begin{equation*}
        R_n^{(b)}=(0,1,b+1,b^2+b+1,b^3+b^2+b+1,\ldots).
    \end{equation*}
    \item When $s=P,t=-Q$, then $\brk[c]{n}_{P,-Q}=U_n(P,Q)$, where $U_n(P,Q)$ is the Lucas sequence, with $P,Q$ integer numbers,
    \begin{equation*}
        U_n(P,Q)=(0,1,P,P^2-Q,P^3-2PQ,P^4-3P^2Q+Q^2,\ldots).
    \end{equation*}
    If $Q=-1$, then the sequence $U_n(P,-1)$ reduces to the $P$-Fibonacci sequence. If $s=x$ and $t=1$, we obtain the Fibonacci polynomials
    \begin{equation*}
        F_n(x)=(0,1,x,x^2+1,x^3+2x,x^4+3x^2+1,\ldots).
    \end{equation*}
\end{enumerate}

The $(s,t)$-Fibonacci constant is the ratio toward which adjacent $(s,t)$-Fibonacci polynomials tend. This is the only positive root of $x^{2}-sx-t=0$. We will let $\varphi_{s,t}$ denote this constant, where
\begin{equation*}
    \varphi_{s,t}=\frac{s+\sqrt{s^{2}+4t}}{2}
\end{equation*}
and
\begin{equation*}
    \varphi_{s,t}^{\prime}=s-\varphi_{s,t}=-\frac{t}{\varphi_{s,t}}=\frac{s-\sqrt{s^{2}+4t}}{2}.
\end{equation*}
Some specializations of the constants $\varphi_{s,t}$ and $\varphi_{s,t}^\prime$ are:
\begin{enumerate}
    \item When $s=0$ and $t=0$, then $\varphi_{0,0}=0$ and $\varphi_{0,0}^\prime=0$.
    \item When $s=0$ and $t>0$, then $\varphi_{0,t}=\sqrt{t}$ and $\varphi_{0,t}^\prime=-\sqrt{t}$.
    \item When $s\neq0$ and $t=0$, then $\varphi_{s,0}=s$ and $\varphi_{s,0}^\prime=0$.
    \item When $s=2$ and $t=-1$, then $\varphi_{2,-1}=1$ and $\varphi_{2,-1}^\prime=1$.
    \item When $s=1$ and $t=1$, then $\varphi_{1,1}=\varphi=\frac{1+\sqrt{5}}{2}$ and $\varphi_{1,1}^\prime=\varphi^\prime=\frac{1-\sqrt{5}}{2}$.
    \item When $s=2$ and $t=1$, then $\varphi_{2,1}=1+\sqrt{2}$ and $\varphi_{2,1}^\prime=1-\sqrt{2}$.
    \item When $s=1$ and $t=2$, then $\varphi_{1,2}=2$ and $\varphi_{1,2}^\prime=-1$.
    \item When $s=p+q$ and $t=-pq$, then $\varphi_{p+q,-pq}=p$ and $\varphi_{p+q,-pq}^\prime=q$.
    \item When $s=2t$ and $t=-1$, then $\varphi_{2t,-1}=t+\sqrt{t^2-1}$ and $\varphi_{2t,-1}^\prime=t-\sqrt{t^2-1}$.
    \item When $s=3$ and $t=-2$, then $\varphi_{3,-2}=2$ and $\varphi_{3,-2}^\prime=1$.
    \item When $s=b+1$ and $t=-b$, then $\varphi_{b+1,-b}=b$ and $\varphi_{b+1,-b}^\prime=1$.
    \item When $s=P$ and $t=-Q$, then $\varphi_{P,-Q}=\frac{P+\sqrt{P^2-4Q}}{2}$ and $\varphi_{P,-Q}^\prime=\frac{P-\sqrt{P^2-4Q}}{2}$.
\end{enumerate}

In this paper we are not interested in all values that the parameters $s,t$ can take, but rather those for which it is satisfied that $s,t\in\R$ such that $s\neq0$ and $s^{2}+4t>0$, because in this way $\brk[c]{n}_{s,t}\in\R[s,t]$, the ring of polynomials in the variables $s,t$ and coefficients in $\R$. In the remainder of the paper we will assume that $q=\varphi_{s,t}^{\prime}/\varphi_{s,t}$. If $t>0$ and $s\neq0$, then $q<0$. If $t<0$ and $s>0$, then $0<q<1$. If $t<0$ and $s<0$, then $q>1$.

The Binet's $(s,t)$-identity is
\begin{equation}\label{eqn_binet}
    \brk[c]{n}_{s,t}=
    \begin{cases}
    \frac{\varphi_{s,t}^{n}-\varphi_{s,t}^{\prime n}}{\varphi_{s,t}-\varphi_{s,t}^{\prime}},&\text{ if }s\neq\pm2i\sqrt{t};\\
    n(\pm i\sqrt{t})^{n-1},&\text{ if }s=\pm2i\sqrt{t}.
    \end{cases}
\end{equation}
As $\varphi_{us,u^2t}=u\varphi_{s,t}$ and $\varphi_{us,u^2t}^{\prime}=u\varphi_{s,t}^{\prime}$, then $\varphi_{s,t}$ and $\varphi_{s,t}^{\prime}$ are quasi-homogeneous functions of type $(u,u^2)$ and follows that $\brk[c]{n}_{us,u^2t}=u^{n-1}\brk[c]{n}_{s,t}$. Thus, we have the following definition.

\begin{definition}
For a non-zero complex number $u$ define the deformed Fibonacci polynomials as $\brk[c]{n}_{us,u^2t}=u^{n-1}\brk[c]{n}_{s,t}$, for all $n\geq1$.
\end{definition}
For example, with $u=\frac{1}{2}$ we can obtain the deformed Fibonacci sequence $(1/2)^{n-1}F_n$, i.e.,
\begin{equation*}
    \brk[c]{n}_{1/2,1/4}=\frac{1}{2^{n-1}}\brk[c]{n}_{1,1}=(0,1,1/2,1/2,3/8,5/16,1/4,13/64,\ldots)
\end{equation*}
generated by the recurrence equation
\begin{equation*}
    \brk[c]{n+2}_{1/2,1/4}=\frac{1}{2}\brk[c]{n+1}_{1/2,1/4}+\frac{1}{4}\brk[c]{n}_{1/2,1/4}.
\end{equation*}
In general, the deformed sequence $\brk[c]{n}_{su,tu^2}$ satisfies the recurrence relation
\begin{equation}\label{eqn_fib_defor}
    \brk[c]{n+2}_{su,tu^2}=su\brk[c]{n+1}_{su,tu^2}+tu^2\brk[c]{n}_{su,tu^2},
\end{equation}
so the deformed Fibonacci polynomials $\brk[c]{n}_{su,tu^2}$ are associated with the characteristic polynomial $p_{s,t}(x,u)=x^2-sux-tu^2$ whose ratios are $\varphi_{su,tu^2}=u\varphi_{s,t}$ and $\varphi_{su,tu^2}^{\prime}=u\varphi_{s,t}^{\prime}$.

\begin{proposition}
For all $s,t\in\R$
\begin{enumerate}
    \item The deformed $(s,t)$-Fibotorial is
    \begin{equation}\label{eqn_fibotorial}
    \brk[c]{n}_{us,u^2t}!=u^{\binom{n}{2}}\prod_{k=1}^{n}\brk[c]{k}_{s,t}=u^{\binom{n}{2}}\brk[c]{n}_{s,t}!,\ n\geq1,\ \brk[c]{0}_{s,t}!=1.
\end{equation}
\item The deformed $(s,t)$-Fibonomial polynomials
\begin{equation}\label{eqn_fibonomial}
    \fibonomial{n}{k}_{us,u^2t}=u^{k(n-k)}\fibonomial{n}{k}_{s,t}.
\end{equation}    
\end{enumerate}
\end{proposition}
Some deformed $(s,t)$-Fibonomial polynomials are
\begin{align*}
    \fibonomial{n}{0}_{us,u^2t}&=\fibonomial{n}{n}_{us,u^2t}=1,\\
    \fibonomial{2}{1}_{us,u^2t}&=su,\\
    \fibonomial{3}{1}_{us,u^2t}&=\fibonomial{3}{2}_{us,u^2t}=(s^2+t)u^2,\\
    \fibonomial{4}{1}_{us,u^2t}&=\fibonomial{4}{3}_{us,u^2t}=(s^3+2st)u^3,\ \fibonomial{4}{2}_{us,u^2t}=(s^2+t)(s^2+2t)u^4.
\end{align*}
For extreme cases $(s,0)$ and $(0,t)$, $t>0$, we have
\begin{align*}
    \brk[c]{n}_{us,0}!&=(us)^{\binom{n}{2}},\\
    \brk[c]{n}_{0,u^2t}!&=u^{\binom{n}{2}}\brk[c]{1}_{0,t}\brk[c]{2}_{0,t}\cdots\brk[c]{n}_{0,t}\\
    &=u^{\binom{n}{2}}(t^0)(0)(t)(0)\cdots=0.
\end{align*}
and $\fibonomial{n}{k}=s^{k(n-k)}$. Then the case $s=0$ and $t\neq0$, nor any other case with $\brk[c]{n}_{s,t}=0$ for some $n\in\N$, will be of interest in this paper.

From Eq. (\ref{eqn_binet}) we obtain that
\begin{equation*}
    \brk[c]{n}_{s,t}=\varphi_{s,t}^{n-1}\frac{1-\left(\frac{\varphi_{s,t}^{\prime}}{\varphi_{s,t}}\right)^{n}}{1-\left(\frac{\varphi_{s,t}^{\prime}}{\varphi_{s,t}}\right)}.
\end{equation*}
If we set $q=\frac{\varphi_{s,t}^{\prime}}{\varphi_{s,t}}$, then
\begin{equation}\label{eqn_fibo_q}
    \brk[c]{n}_{s,t}=\varphi_{s,t}^{n-1}\frac{1-q^n}{1-q}=\varphi_{s,t}^{n-1}\brk[s]{n}_q=\varphi_{s,t}^{\prime(n-1)}\brk[s]{n}_{q^{-1}},
\end{equation}
where $\brk[s]{n}_q=\frac{1-q^n}{1-q}$, from which the relationship between $(s,t)$-Fibonacci polynomials and the $q$-numbers is clear. The Eq.(\ref{eqn_fibo_q}) implies the following identities
\begin{eqnarray}
    \brk[c]{n}_{s,t}!&=&\varphi_{s,t}^{\binom{n}{2}}\brk[s]{n}_q!,\label{eqn_fibo_fact}=\varphi_{s,t}^{\prime\binom{n}{2}}\brk[s]{n}_{q^{-1}}!,\\
    \fibonomial{n}{k}_{s,t}&=&\varphi_{s,t}^{k(n-k)}\binom{n}{k}_q=\varphi_{s,t}^{\prime k(n-k)}\binom{n}{k}_{q^{-1}}.\label{eqn_fibo_bin}
\end{eqnarray}
When $u=\varphi_{s,t}$, then
\begin{align*}
    \varphi_{s,t}^{n-1}\brk[s]{n}_q&=\varphi_{s,t}^{n-1}\brk[c]{n}_{1+q,-q}\\
    &=\brk[c]{n}_{(1+q)\varphi,-q\varphi^2}\\
    &=\brk[c]{n}_{\varphi+\varphi^{\prime},-\varphi\varphi^{\prime}}\\
    &=\brk[c]{n}_{s,t}
\end{align*}
and accordingly $\brk[c]{n}_{s,t}$ is a deformation of $\brk[s]{n}_q$.

On the other hand, the Chebysheff polynomial of the second kind $U_n(t)$ is a polynomial of degree $n$ in the variable $t$ defined by
\begin{equation*}
U_n(t)=\frac{\sin(n+1)\theta}{\sin\theta}
\end{equation*}
where $t=\cos\theta$. The polynomials $U_n(t)$ satisfy the recurrence relation
\begin{equation*}
U_n(t)=2tU_{n-1}(t)-U_{n-2}(t)
\end{equation*}
together with the initial conditions $U_0(t)=1$ and $U_1(t)=2t$. Then the polynomials $U_n(t)$ are $(2t,-1)$-Fibonacci polynomials with $\brk[c]{n}_{2t,-1}=U_{n-1}(t)$, where $\brk[c]{0}_{2t,-1}=U_{-1}(t)=0$. Then an explicit expression for Chebysheff polynomials of the second kind is through Binet's form
\begin{equation}\label{eqn_fibo_che}
    U_{n-1}(t)=\frac{(t+\sqrt{t^{2}-1})^{n}-(t-\sqrt{t^{2}-1})^{n}}{2\sqrt{t^{2}-1}}.
\end{equation}
Then we can use this representation to express the polynomials $\brk[c]{n}_{s,t}$ in terms of Chebysheff polynomials of the second kind.
For $t\neq0$ and $s\neq\pm2i\sqrt{t}$, the $(s,t)$-Fibonacci polynomials are related to the Chebysheff polynomials in the following way \cite{udrea}
\begin{align}\label{theo_fibo_che}
    \left(\frac{i}{\sqrt{t}}\right)^{n-1}\brk[c]{n}_{s,t}&=U_{n-1}\left(\frac{is}{2\sqrt{t}}\right).
\end{align}
For $s\neq0$ and $t\neq0$ let $\theta_{s,t}$ denote the function
\begin{equation*}
    \theta_{s,t}=\arccos\left(\frac{is}{2\sqrt{t}}\right).
\end{equation*}
Then a trigonometric expression for $(s,t)$-Fibonacci polynomials is
\begin{equation*}
    \brk[c]{n}_{s,t}=(-i\sqrt{t})^{n-1}\frac{\sin(n\theta_{s,t})}{\sin(\theta_{s,t})}=\frac{2(-i)^{n-1}(\sqrt{t})^{n}}{\sqrt{s^{2}+4t}}\sin(n\theta_{s,t}),
\end{equation*}
where it is clear that $\sin(\theta_{s,t})\neq0$ provided that $s\neq0$ and  $s^2+4t>0$. Also, the $(s,t)$-Fibotorial and the $(s,t)$-Fibonomial polynomials can be expressed as
\begin{align*}
    \brk[c]{n}_{s,t}!&=(-i\sqrt{t})^{\binom{n}{2}}U_{n-1}(t)!=(-i\sqrt{t})^{\binom{n}{2}}\frac{\prod_{k=1}^{n}\sin(k\theta_{s,t})}{\sin^{n}(\theta_{s,t})}\label{eqn_fibotorial_che}
\end{align*}
and
\begin{align*}
    \fibonomial{n}{k}_{s,t}&=(-i\sqrt{t})^{k(n-k)}\frac{\prod_{j=k+1}^{n}\sin(j\theta_{s,t})}{\prod_{j=1}^{n-k}\sin(j\theta_{s,t})}
\end{align*}
respectively. Set $u=i/\sqrt{t}$ and $\tau=is/2\sqrt{t}$, then
\begin{align*}
    (i/\sqrt{t})^{-n+1}\brk[c]{n}_{2\tau,-1}&=(i/\sqrt{t})^{-n+1}\brk[c]{n}_{2(is/2\sqrt{t}),-1}\\
    &=\brk[c]{n}_{2(is/2\sqrt{t})(i/\sqrt{t})^{-1},-(i/\sqrt{t})^{-2}}\\
    &=\brk[c]{n}_{s,t}
\end{align*}
and thus $\brk[c]{n}_{s,t}$ is a deformation of $\brk[c]{n}_{2\tau,-1}$. Then, by deformation of $\brk[c]{n}_{s,t}$ we obtain Chebysheff polynomials. Moreover, it follows from Eq. (\ref{eqn_binet}) that $\brk[c]{n}_{\pm2i\sqrt{t},t}$ is a deformation of $n$.

\begin{proposition}\label{prop_abs_nst}
For $s\neq0$ and for all $t\in\R$ such that $s^2+4t>0$,
\begin{enumerate}
    \item  $\vert\brk[c]{n}_{s,t}\vert=\brk[c]{n}_{\vert s\vert,t}$.
    \item $\vert\brk[c]{n}_{s,t}!\vert=\brk[c]{n}_{\vert s\vert,t}!$.
    \item $\vert\fibonomial{n}{k}_{s,t}\vert=\fibonomial{n}{k}_{\vert s\vert,t}$.
\end{enumerate}
\end{proposition}
\begin{proof}
Follow easily if $s>0$ and $t>0$. If $s>0$ and $t<0$, then $0<\varphi_{s,t}^\prime<\varphi_{s,t}$ and $\vert\brk[c]{n}_{s,t}\vert=\brk[c]{n}_{s,t}$. If $s<0$ and $t\neq0$, then for $s=-u$, $u>0$
\begin{align*}
    \varphi_{s,t}&=\frac{-u+\sqrt{u^2+4t}}{2}=-\frac{u-\sqrt{u^2+4t}}{2}=-\varphi_{-s,t}^\prime,\\
    \varphi_{s,t}^{\prime}&=\frac{-u-\sqrt{u^2+4t}}{2}=-\frac{u+\sqrt{u^2+4t}}{2}=-\varphi_{-s,t}.
\end{align*}
In this way
\begin{align*}
\brk[c]{n}_{s,t}&=\frac{\varphi_{s,t}^{n}-\varphi_{s,t}^{\prime n}}{\varphi_{s,t}-\varphi_{s,t}^{\prime}}=\frac{(-\varphi_{-s,t}^{\prime})^n-(-\varphi_{-s,t})^n}{-\varphi_{-s,t}^\prime+\varphi_{-s,t}}=(-1)^{n+1}\brk[c]{n}_{-s,t}
\end{align*}
and
\begin{align*}
    \vert\brk[c]{n}_{s,t}\vert=\vert(-1)^{n+1}\brk[c]{n}_{-s,t}\vert=\vert\brk[c]{n}_{-s,t}\vert=\brk[c]{n}_{-s,t}=\brk[c]{n}_{\vert s\vert,t}
\end{align*}
and thus we obtain the first result. The other statements follow from here.
\end{proof}

\subsection{$q$-numbers as "bifurcation" of $(s,t)$-numbers}

\begin{theorem}\label{theo_lim_seq}
Set $q=\varphi_{s,t}^{\prime}/\varphi_{s,t}$ and take $u\in\R$.
\begin{enumerate}
    \item If $0<\vert q\vert<1$, then
    \begin{equation*}
        \lim_{n\rightarrow\infty}u^{n-1}\brk[c]{n}_{s,t}=
        \begin{cases}
            0,&\text{ if }0\leq u<\varphi_{s,t}^{-1};\\
            \frac{1}{1-q},&\text{ if } u=\varphi_{s,t}^{-1};\\
            \infty,&\text{ if } u>\varphi_{s,t}^{-1}.
        \end{cases}
    \end{equation*}
    \item If $\vert q\vert>1$, then
    \begin{equation*}
        \lim_{n\rightarrow\infty}u^{n-1}\brk[c]{n}_{s,t}=
        \begin{cases}
            0,&\text{ if }0\leq u<\varphi_{s,t}^{\prime-1};\\
            \frac{1}{1-q^{-1}},&\text{ if }u=\varphi_{s,t}^{\prime-1};\\
            \infty,&\text{ if }u>\varphi_{s,t}^{\prime-1}.
        \end{cases}
    \end{equation*}
\end{enumerate}
\end{theorem}
\begin{proof}
From Eq.(\ref{eqn_fibo_q}), 
\begin{equation*}
    u^{n-1}\brk[c]{n}_{s,t}=(u\varphi_{s,t})^{n-1}\brk[s]{n}_{q}=(u\varphi_{s,t}^{\prime})^{n-1}\brk[s]{n}_{q^{-1}},
\end{equation*}
with $q=\varphi_{s,t}^{\prime}/\varphi_{s,t}$. Now we apply limits when $n$ tends to $\infty$ and take into account the given conditions.
\end{proof}
We can write the deformed generalized Fibonacci polynomials, Eq.(\ref{eqn_fib_defor}), as a 2-dimensional system 
\begin{equation*}
    \left(
    \begin{array}{c}
         \brk[c]{n+2}_{us,u^2t}\\
         \brk[c]{n+1}_{us,u^2t}
    \end{array}
    \right)=
    \left(
    \begin{array}{cc}
         us&u^{2}t  \\
         1&0 
    \end{array}
    \right)\left(
    \begin{array}{c}
         \brk[c]{n+1}_{us,u^2t} \\
         \brk[c]{n}_{us,u^2t}
    \end{array}
    \right)
\end{equation*}
and in vector form
\begin{equation}\label{eqn_sys_gfp}
    \overset{\rightarrow}{F}_{n+1}=A\overset{\rightarrow}{F}_{n},
\end{equation}
which yields $\overset{\rightarrow}{F}_{n}=A^{n}\overset{\rightarrow}{F}_{0}$. The eigenvalues of the matrix $A$ are $u\varphi_{s,t}$ and $u\varphi_{s,t}^{\prime}$ corresponding to the respective eigenvectors
\begin{equation*}
    \overset{\rightarrow}{\mu}=
    \left(
    \begin{array}{c}
         u\varphi_{s,t} \\
         1
    \end{array}
    \right)
\end{equation*}
and
\begin{equation*}
    \overset{\rightarrow}{\nu}=
    \left(
    \begin{array}{c}
         u\varphi_{s,t}^{\prime}\\
         1
    \end{array}
    \right).
\end{equation*}
As the initial value is
\begin{equation*}
    \overset{\rightarrow}{F}_{0}=\frac{1}{u\sqrt{s^{2}+4t}}\overset{\rightarrow}{\mu}-\frac{1}{u\sqrt{s^{2}+4t}}\overset{\rightarrow}{\nu},
\end{equation*}
it follows that the $n$-th term is
\begin{equation*}
    \overset{\rightarrow}{F}_{n}=\frac{(u\varphi_{s,t})^{n}}{u\sqrt{s^{2}+4t}}\overset{\rightarrow}{\mu}-\frac{(u\varphi_{s,t}^{\prime})^n}{u\sqrt{s^{2}+4t}}\overset{\rightarrow}{\nu}.
\end{equation*}

\begin{theorem}\label{theo_matrixn}
For all $u\in\C$ and $n\in\N$, 
\begin{equation*}
    \left(
    \begin{array}{cc}
     us  &  u^{2}t\\
      1   & 0
    \end{array}
    \right)^{n}
    =
    \left(
    \begin{array}{cc}
        u^{n}\brk[c]{n+1}_{s,t} &u^{n+1}t\brk[c]{n}_{s,t}  \\
        u^{n-1}\brk[c]{n}_{s,t} & u^{n}t\brk[c]{n-1}_{s,t}
    \end{array}
    \right)
\end{equation*} 
\end{theorem}
\begin{proof}
The proof is by induction on $n$. The result is valid for $n=1$. Suppose true for $n$ and let us prove for $n+1$. We have that
\begin{align*}
    \left(
    \begin{array}{cc}
     us  &  u^{2}t\\
      1   & 0
    \end{array}
    \right)^{n+1}
    &=
    \left(
    \begin{array}{cc}
     us  &  u^{2}t\\
      1   & 0
    \end{array}
    \right)^{n}
    \left(
    \begin{array}{cc}
     us  &  u^{2}t\\
      1   & 0
    \end{array}
    \right)\\
    &=
    \left(
    \begin{array}{cc}
        u^{n}\brk[c]{n+1}_{s,t} &u^{n+1}t\brk[c]{n}_{s,t}  \\
        u^{n-1}\brk[c]{n}_{s,t} & u^{n}t\brk[c]{n-1}_{s,t}
    \end{array}
    \right)
    \left(
    \begin{array}{cc}
     us  &  u^{2}t\\
      1   & 0
    \end{array}
    \right)\\
    &=\left(
    \begin{array}{cc}
        u^{n+1}(s\brk[c]{n+1}_{s,t}+t\brk[c]{n}_{s,t}) &u^{n+2}t\brk[c]{n+1}_{s,t}  \\
        u^{n}(s\brk[c]{n}_{s,t}+t\brk[c]{n-1}_{s,t}) & u^{n+1}t\brk[c]{n}_{s,t}
    \end{array}
    \right)\\
    &=\left(
    \begin{array}{cc}
        u^{n+1}\brk[c]{n+2}_{s,t} &u^{n+2}t\brk[c]{n+1}_{s,t}  \\
        u^{n}\brk[c]{n+1}_{s,t} & u^{n+1}t\brk[c]{n}_{s,t}
    \end{array}
    \right)
\end{align*}
and the statement turns out to be true for all $n$.
\end{proof}
From Theorems \ref{theo_lim_seq} and \ref{theo_matrixn} we have the following results.
\begin{theorem}
For all $u\geq0$,
\begin{equation*}
    \lim_{n\rightarrow\infty}
    \left(
     \begin{array}{cc}
          us& u^{2}t \\
          1&0 
     \end{array}   
    \right)^{n}
    =
    \begin{cases}
     \left(
     \begin{array}{cc}
         0 & 0 \\
         0 & 0
     \end{array}   
     \right),&\text{ if }0\leq u<\varphi_{s,t}^{-1} \text{ or }0\leq u<\varphi_{s,t}^{\prime-1};\\
     \frac{1}{\varphi_{s,t}-\varphi_{s,t}^{\prime}}
     \left(
     \begin{array}{cc}
         \varphi_{s,t} & t\varphi_{s,t}^{-1} \\
         \varphi_{s,t} & t\varphi_{s,t}^{-1}
     \end{array}   
     \right),
     &\text{ if }u=\varphi_{s,t}^{-1};\\
     \frac{-1}{\varphi_{s,t}-\varphi_{s,t}^{\prime}}
     \left(
     \begin{array}{cc}
         \varphi_{s,t}^{\prime} & t\varphi_{s,t}^{\prime-1} \\
         \varphi_{s,t}^{\prime} & t\varphi_{s,t}^{\prime-1}
     \end{array}   
     \right),
     &\text{ if }u=\varphi_{s,t}^{\prime-1}.
    \end{cases}
\end{equation*}    
\end{theorem}
From all of the above we can conclude that
\begin{equation*}
    \lim_{n\rightarrow\infty}\overset{\rightarrow}{F}_{n}=\lim_{n\rightarrow\infty}A^{n}\overset{\rightarrow}{F}_{0}
    =
    \begin{cases}
    (0,0)^{\mathrm{T}},&\text{ if }0\leq u<\varphi_{s,t}^{-1} \text{ or if }0\leq u<\varphi_{s,t}^{\prime-1};\\
    \left(\frac{1}{1-q},\frac{1}{1-q}\right)^{\mathrm{T}},&\text{ if }u=\varphi_{s,t}^{-1};\\
    \left(\frac{1}{1-q^{-1}},\frac{1}{1-q^{-1}}\right)^{\mathrm{T}},&\text{ if }u=\varphi_{s,t}^{\prime-1};\\
    (\infty,\infty)^{\mathrm{T}},&\text{ if }u>\varphi_{s,t}^{-1}\text{ or if }u>\varphi_{s,t}^{\prime-1}.
    \end{cases}
\end{equation*}
Then $(0,0)$, $\left(\frac{1}{1-q},\frac{1}{1-q}\right)$ and $\left(\frac{1}{1-q^{-1}},\frac{1}{1-q^{-1}}\right)$ are asymptotically stable fixed points of Eq.(\ref{eqn_sys_gfp}), provided that either $u\in[0,\varphi_{s,t}^{-1}]$ or $u\in[0,\varphi_{s,t}^{\prime-1}]$. If either $u>\varphi_{s,t}^{-1}$ or $u>\varphi_{s,t}^{\prime-1}$, the the solution $\overset{\rightarrow}{F}_{n}$ remains unbounded, i.e., $\vert\overset{\rightarrow}{F}_{n}\vert\rightarrow\infty$ as $n\rightarrow\infty$. Then $u=\varphi_{s,t}^{-1}$ and $u=\varphi_{s,t}^{\prime-1}$ are bifurcation points of Eq.(\ref{eqn_sys_gfp}). At these points, $u^{n-1}\brk[c]{n}_{s,t}=\brk[s]{n}_{q}$ and $u^{n-1}\brk[c]{n}_{s,t}=\brk[s]{n}_{q^{-1}}$, respectively, and thus the $q$-numbers become "bifurcation points" of the family of deformed $(s,t)$-numbers.

\section{Calculus on deformed generalized Fibonacci polynomials}

\subsection{The Ward ring of Fibonomial exponential generating functions}

Now let $\ward_{s,t,\C}[[x]]$ denote the set of $(s,t)$-exponential generating functions of the form $\sum_{n=0}^{\infty}a_n(x^{n}/\brk[c]{n}_{s,t}!)$ with coefficients in $\C$. It is clear that $(\ward_{s,t,\C}[[x]],+,\cdot)$ is a ring with sum and product ordinary of series, that is,
$$f(x)+g(x)=\sum_{n=0}^{\infty}(a_n+b_n)\frac{x^{n}}{\brk[c]{n}_{s,t}!}$$
and
$$f(x)\cdot g(x)=\sum_{n=0}^{\infty}\sum_{k=0}^{n}\fibonomial{n}{k}_{s,t}a_{k}b_{n-k}\frac{x^{n}}{\brk[c]{n}_{s,t}!},$$ 
where $f(x)=\sum_{n=0}^{\infty}a_n(x^{n}/\brk[c]{n}_{s,t}!)$, 
$g(x)=\sum_{n=0}^{\infty}b_n(x^{n}/\brk[c]{n}_{s,t}!)\in\ward_{s,t,\C}[[x]]$. 
\begin{definition}\label{defi_st_ward}
The ring $\ward_{s,t,\C}[[z]]$ will be called generalized Ward-Fibonomial ring of $(s,t)$-exponential generating functions, or $(s,t)$-Ward ring,
\begin{equation*}
    \sum_{n=0}^{\infty}a_n\frac{z^n}{\brk[c]{n}_{s,t}!}
\end{equation*}
with coefficient in $\C$. Let $f_{s,t}(z,u)$ denote the function
\begin{equation*}
    f_{s,t}(z,u)=\sum_{n=0}^{\infty}u^{\binom{n}{2}}a_{n}z^{n}/\brk[c]{n}_{s,t}!
\end{equation*}
in $\ward_{s,t,\C}[[z]]$ We will call to $f_{s,t}(z,u)$ a deformed function.
\end{definition}

\begin{theorem}\label{theo_conv_st}
Set $s,t\in\R$ such that $s\neq0$, $s^2+4t>0$ and $(s,t)\neq(2,-1)$. Take $f_{s,t}(z,u)$ in $\ward_{s,t,\C}[[z]]$, such that $\brk[c]{n}_{s,t}!\nmid a_{n}$ and $u^{-\binom{n}{2}}\nmid a_{n}$ for all $n\geq0$. Set $\alpha=\lim_{n\rightarrow\infty}\vert a_{n+1}/a_{n}\vert\geq0$.
If $\vert q\vert<1$, the function $f_{s,t}(z,u)$ is:
\begin{enumerate}
    \item[1.] An entire function if $\vert u\vert<\vert\varphi_{s,t}\vert$.
    \item[2.] Convergent in the disk $\vert z\vert<1/\alpha\vert1-q\vert$ when $\alpha>0$ and an entire function when $\alpha=0$, provided that $\vert u\vert=\vert\varphi_{s,t}\vert$.
    \item[3.] Convergent in $z=0$ when $\vert u\vert>\vert\varphi_{s,t}\vert$.
\end{enumerate}
Suppose that $\vert q\vert>1$. The function $f_{s,t}(z,u)$ is:
\begin{enumerate}
    \item[4.] An entire when $\vert u\vert<\vert\varphi_{s,t}^\prime\vert$.
    \item[5.] Convergent in the disk $\vert z\vert<1/\alpha\vert q^{-1}-1\vert$ when $\alpha>0$ and an entire function when $\alpha=0$, provided that $\vert u\vert=\vert\varphi_{s,t}^\prime\vert$.
    \item[6.] Convergent in $z=0$ when $\vert u\vert>\vert\varphi_{s,t}^\prime\vert$.
\end{enumerate}
\end{theorem}
\begin{proof}
We will use the D'alembert test. If $\vert q\vert<1$, then
\begin{align*}
    \lim_{n\rightarrow\infty}\Bigg\vert\frac{xu^{n}a_{n+1}}{a_{n}\brk[c]{n+1}_{s,t}!}\Bigg\vert&=\alpha\vert z\vert\lim_{n\rightarrow\infty}\bigg\vert\frac{u^n}{\brk[c]{n+1}_{s,t}}\bigg\vert\\
    &=\alpha\vert z\vert\lim_{n\rightarrow\infty}\Bigg\vert\frac{(\varphi_{s,t}-\varphi_{s,t}^{\prime})u^{n}}{\varphi_{s,t}^{n+1}-\varphi_{s,t}^{\prime(n+1)}}\Bigg\vert\\
    &=\alpha\vert z\vert\vert1-q\vert\lim_{n\rightarrow\infty}\bigg\vert\frac{1}{1-q^{n+1}}\bigg\vert\bigg\vert\frac{u^n}{\varphi_{s,t}^n}\bigg\vert.
\end{align*}
Now, follow the first 3 statements. Now suppose that $\vert q\vert>1$. Then
\begin{align*}
    \lim_{n\rightarrow\infty}\Bigg\vert\frac{xu^{n}a_{n+1}}{a_{n}\brk[c]{n+1}_{s,t}!}\Bigg\vert&=\alpha\vert z\vert\vert q^{-1}-1\vert\lim_{n\rightarrow\infty}\bigg\vert\frac{1}{q^{-(n+1)}-1}\bigg\vert\bigg\vert\frac{u^n}{\varphi_{s,t}^{\prime n}}\bigg\vert.
\end{align*}
Therefore follow the last three statements.
\end{proof}

The above theorem is consistent with Theorem \ref{theo_lim_seq}, in the sense that the convergence of a deformed function $f_{s,t}(z,u)$ depends strongly on the values taken by the parameter $u$. In the following theorem, we analyze the case $(s,t)=(2,-1)$.

\begin{theorem}\label{theo_conv_2-1}
Let $f_{2,-1}(z,u)$ denote the function $f_{2,-1}(z,u)=\sum_{n=0}^{\infty}u^{\binom{n}{2}}a_{n}z^{n}/n!$ with $a_{n}\in\C$ such that $n!\nmid a_{n}$ and $u^{-\binom{n}{2}}\nmid a_{n}$ for all $n\geq0$. Set $\alpha=\lim_{n\rightarrow\infty}\vert a_{n+1}/a_{n}\vert$. Then $f_{2,-1}(z,u)$ is an entire function when $0<u\leq1$ and convergent in $z=0$ when $u>1$.
\end{theorem}

\begin{example}
The function $\frac{1}{1-z}=\sum_{n=0}^{\infty}z^{n}$ is convergent when $\vert z\vert<1$. Then by the above theorem, $\Exp(z,u)=\sum_{n=0}^{\infty}u^{\binom{n}{2}}\frac{z^n}{n!}$ is an entire function when $0<u\leq1$.
\end{example}

Because of what we have seen in the previous section and because of Theorems \ref{theo_conv_st} and \ref{theo_conv_2-1}, in the remainder of the paper we will restrict ourselves only to deformed functions with deformation parameter $u\in(0,\vert\varphi_{s,t}\vert]$ and $u\in(0,\vert\varphi_{s,t}^{\prime}\vert]$. 

\subsection{The deformed $(s,t)$-derivative}

In this section, we introduce the deformed differential operator. In addition, we introduce the set of $q$-periodic functions and it will be shown that these functions are invariant by deformations.

\begin{definition}
Set $s\neq0$ and $s^2+4t>0$. For all $u\in(0,\vert\varphi_{s,t}\vert]$ or $u\in(0,\vert\varphi_{s,t}^\prime\vert]$ define the deformed $(s,t)$-derivative $\mathbf{D}_{su,tu^2}$ of the function $f(x)$ as
\begin{equation*}
(\mathbf{D}_{su,tu^2}f)(x)=\frac{f(u\varphi_{s,t}x)-f(u\varphi_{s,t}^{\prime}x)}{u(\varphi_{s,t}-\varphi_{s,t}^{\prime})x}
\end{equation*}
for all $x\neq0$ and $(\mathbf{D}_{us,u^2t}f)(0)=f^{\prime}(0)$, provided $f^{\prime}(0)$ exists.
\end{definition}
The deformed $(s,t)$-derivative is a particular case of the $(p,q)$-derivative with $p=u\varphi_{s,t}$ and $q=u\varphi_{s,t}^{\prime}$. Next, we will give the definitions of deformed derivatives according to each specialization.

\begin{definition}
Take $s=1,\ t=1$. For all $u\in(0,\frac{1+\sqrt{5}}{2}]$ or $u\in(0,\vert\frac{1-\sqrt{5}}{2}\vert]$, we define the deformed Fibonacci derivative
\begin{equation*}
    \mathbf{D}_{u,u^2}f(x)=\mathbf{D}_{F,u}f(x)=
    \begin{cases}
        \frac{f\left(u\varphi x\right)-f\left(u\varphi^{\prime}x\right)}{u\sqrt{5}x},&\text{ if }x\neq0;\\
        f^{\prime}(0), &\text{ if } x=0.
    \end{cases}
\end{equation*}
\end{definition}

\begin{definition}
Take $s=2,\ t=1$. For all $u\in(0,1+\sqrt{2}]$ or $u\in(0,\vert1-\sqrt{2}\vert]$ we define the deformed Pell derivative
\begin{equation*}
    \mathbf{D}_{2u,u^2}f(x)=\mathbf{D}_{P,u}f(x)=
    \begin{cases}
        \frac{f\left(u(1+\sqrt{2})x\right)-f\left(u(1-\sqrt{2})x\right)}{2u\sqrt{2}x},&\text{ if }x\neq0;\\
        f^{\prime}(0), &\text{ if } x=0.
    \end{cases}
\end{equation*}
\end{definition}

\begin{definition}
Take $s=1,\ t=2$. For all $u\in(0,2]$ or $u\in(0,1]$ we define the deformed Jacobsthal derivative
\begin{equation*}
    \mathbf{D}_{u,2u^2}f(x)=\mathbf{D}_{J,u}f(x)=
    \begin{cases}
    \frac{f\left(2ux\right)-f\left(-ux\right)}{3ux},&\text{ if }x\neq0;\\
    f^{\prime}(0),&\text{ if }x=0.
    \end{cases}
\end{equation*}
\end{definition}

\begin{definition}
Take $s=3,\ t=-2$. For all $u\in(0,2]$ or $u\in(0,1]$ we define the deformed Mersenne derivative
\begin{equation*}
    \mathbf{D}_{u,2u^2}f(x)=\mathbf{D}_{M,u}f(x)=
    \begin{cases}
    \frac{f\left(2ux\right)-f\left(ux\right)}{ux},&\text{ if }x\neq0;\\
    f^{\prime}(0),&\text{ if }x=0.
    \end{cases}
\end{equation*}
\end{definition}

\begin{definition}
Take $s=p+q,\ t=-pq$. For all $u\in(0,\vert p\vert]$ or $u\in(0,\vert q\vert]$ we define the deformed $(p,q)$-derivative
\begin{equation*}
    \mathbf{D}_{u(p+q),-pqu^2}f(x)=D_{p,q,u}f(x)=
    \begin{cases}
    \frac{f\left(upx\right)-f\left(uqx\right)}{u(p-q)x},&\text{ if }x\neq0;\\
    f^{\prime}(0),&\text{ if }x=0.
    \end{cases}
\end{equation*}
\end{definition}

\begin{definition}
Take $s=2r,\lvert r\rvert>1,\ t=-1$. For all $u\in(0,r+\sqrt{r^2-1}]$ or $u\in(0,\vert r-\sqrt{r^2-1}]$ we define the deformed Chebysheff of second kind derivative
\begin{equation*}
    \mathbf{D}_{2ru,-u^2}f(x)=D_{U,u}f(x)=
    \begin{cases}
    \frac{f\left(u(r+\sqrt{r^{2}-1})x\right)-f\left(u(r-\sqrt{r^{2}-1})x\right)}{2u\sqrt{r^{2}-1}x},&\text{ if }x\neq0;\\
    f^{\prime}(0),&\text{ if }x=0.
    \end{cases}
\end{equation*}
\end{definition}

\begin{definition}
Take $s=P,\ t=-Q$. For all 
\begin{equation*}
u\in\left(0,\frac{1}{2}\vert P+\sqrt{P^2-4Q}\vert\right] \text{ or } u\in\left(0,\frac{1}{2}\vert P-\sqrt{P^2-4Q}\vert\right]
\end{equation*}
we define the deformed Lucas-derivative
\begin{equation*}
    \mathbf{D}_{uP,-Qu^2}f(x)=D_{P,-Q,u}f(x)=
    \begin{cases}
    \frac{f\left(u\varphi_{P,-Q}x\right)-f\left(u\varphi_{P,-Q}^{\prime}x\right)}{u(\varphi_{P,-Q}-\varphi_{P,-Q}^{\prime})x},&\text{ if }x\neq0;\\
    f^{\prime}(0),&\text{ if }x=0.
    \end{cases}
\end{equation*}
\end{definition}

\begin{example}
For function $f(x)=x^{5}$, deformed $(s,t)$-derivative is obtained as
\begin{equation*}
    \mathbf{D}_{su,tu^2}x^{5}=\brk[c]{5}_{su,tu^2}x^{4}=u^{4}\brk[c]{5}_{s,t}x^4.
\end{equation*}
Then, $\mathbf{D}_{F,u}x^5=5u^4x^4$, $\mathbf{D}_{P,u}x^5=29u^4x^4$, $\mathbf{D}_{J,u}x^5=5u^4x^4$, $\mathbf{D}_{M,u}x^5=31u^4x^4$ and
\begin{align*}
    D_{U,u}x^{5}&=u^4(16t^4-12t^2+1)x^4,\\
    D_{P,-Q,u}x^{5}&=u^{4}(P^4-3P^2Q+Q^2)x^{4}.
\end{align*}
\end{example}

\begin{example}
The ordinary and exponential generating functions of the sequence $(1,1,\ldots)$ are transformed into the ordinary and exponential generating functions of the polynomials $\brk[c]{n}_{s,t}$, respectively, through the operator in differences $\mathbf{D}_{s,t}$
    \begin{align*}
        x\mathbf{D}_{s,t}\left(\frac{1}{1-x}\right)&=\frac{x}{1-sx-tx^2},\\
        x\mathbf{D}_{s,t}e^{x}&=\frac{e^{\varphi_{s,t}x}-e^{\varphi_{s,t}^{\prime}x}}{\sqrt{s^2+4t}}.
    \end{align*}
\end{example}

For two fixed complex numbers $s,t$ the ring $\ward_{s,t,\C}[[z]]$ is equipped with the family of derivatives $\mathbf{D}_{su,tu^2}$, for every nonzero complex number $u$. In particular, in the ring $\ward_{s,t,\C}[[z]]$, we define the following derivatives
\begin{equation}\label{eqn_qdiff}
    D_q=\mathbf{D}_{1+q,-q}=\mathbf{D}_{s(1/\varphi_{s,t}),t(1/\varphi_{s,t})^2}
\end{equation}
and
\begin{equation}\label{eqn_Udiff}
    D_{U}=\mathbf{D}_{2(is/2\sqrt{t}),-1}=\mathbf{D}_{s(i/\sqrt{t}),t(i/\sqrt{t})^2}.
\end{equation}

The following results on the deformed $(s,t)$-derivative are standard:
\begin{align}
    \mathbf{D}_{su,tu^2}(f(x)+g(x))&=(\mathbf{D}_{su,tu^2}f)(x)+(\mathbf{D}_{su,tu^2}g)(x),\nonumber\\
    \mathbf{D}_{su,tu^2}(\alpha f(x))&=\alpha(\mathbf{D}_{su,tu^2}f)(x),\text{ for all }\alpha\in\C,\nonumber\\
    \mathbf{D}_{su,tu^2}(f(x)g(x))&=f(u\varphi_{s,t}x)(\mathbf{D}_{su,tu^2}g)(x)+g(u\varphi_{s,t}^{\prime}x)(\mathbf{D}_{su,tu^2}f)(x),\label{eqn_der_prod1}\\
    &=f(u\varphi_{s,t}^{\prime}x)(\mathbf{D}_{su,tu^2}g)(x)+g(u\varphi_{s,t}x)(\mathbf{D}_{su,tu^2}f)(x).\label{eqn_der_prod2}    
\end{align}

\begin{definition}
We will say that the function $f(x)$ is $q$-periodic if $f(y)=f(qy)$, with $y=\varphi_{s,t}x$. Let $\Per_{s,t}$ denote the set of  $q$-periodic functions. The $q$-periodic functions form the kernel of the operator $\mathbf{D}_{s,t}$.
\end{definition}
Set $s\neq0$ and $t<0$. From the condition of $q$-periodicity of $f(x)$ it follows that $f(y)=f(qy)$, with $q>0$. Then
\begin{align*}
    f(q^y)&=f(q^{y+1})\\
    G(y)&=G(y+1)
\end{align*}
where $G$ is an arbitrary periodic function with period one and $f(x)=G(\log_{q}(x))$, $x>0$ and $\log_{q}(x)$ is the logarithm function in base $q$. Thus,
\begin{align*}
    \mathbf{D}_{s,t}f(x)&=\frac{G(\log_{q}(\varphi_{s,t}x))-G(\log_{q}(\varphi_{s,t}^{\prime}x))}{(\varphi_{s,t}-\varphi_{s,t}^{\prime})x}\\
    &=\frac{G(\log_{q}(\varphi_{s,t}x))-G(\log_{q}(q\varphi_{s,t}x))}{(\varphi_{s,t}-\varphi_{s,t}^{\prime})x}\\
    &=\frac{G(\log_{q}(\varphi_{s,t}x))-G(\log_{q}(\varphi_{s,t}x))}{(\varphi_{s,t}-\varphi_{s,t}^{\prime})x}=0
\end{align*}
and $f(x)\in\Per_{s,t}$. If $s\neq0$ and $t>0$, then $q<0$ and $f(x)=G(\log(x)/(\log(-q)+i\pi))$, so that $x\in\C/\{0\}$ and $f(x)\in\Per_{s,t}$.
As $q=\frac{\varphi_{su,tu^2}^\prime}{\varphi_{su,tu^2}}=\frac{\varphi_{s,t}^\prime}{\varphi_{s,t}}$, every function in $\Per_{s,t}$ is invariant by deformations.
Thus $\Per_{s,t}=\Per_{su,tu^2}$. Now, set $s\neq0$ and $t=0$. Then $\Per_{s,0}=\C$. 

In the following theorem, we will show that the set of $\varphi_{s,t}$-periodic functions is an algebra.
\begin{theorem}
The set $\Per_{s,t}$ is a $\C$-algebra.
\end{theorem}
\begin{proof}
If $f(x),g(x)\in\Per_{s,t}$, then
\begin{align*}
 \mathbf{D}_{s,t}(f(x)+g(x))&=0\\   
 \mathbf{D}_{s,t}(f(x)g(x))&=0\\
 \mathbf{D}_{s,t}(\alpha f(x))&=0
\end{align*}
for all $\alpha\in\C$. Then $\Per_{s,t}$ is closed concerning sum, product of functions, and product by a scalar. Now the properties of $\C$-algebra follow easily.
\end{proof}
If $f(x)\in\Per_{s,t}$, then
\begin{equation*}
    \mathbf{D}_{s,t}(f(x)g(x))=f(\varphi_{s,t}x)(\mathbf{D}_{s,t}g)(x)=f(\varphi_{s,t}^\prime x)(\mathbf{D}_{s,t}g)(x).
\end{equation*}
In $\ward_{s,t,\C}[[x]]$ the only function satisfying $(\mathbf{D}_{s,t}f)(x)=0$ is the constant function $f(x)=C$. Then $\Per_{s,t}\cap\ward_{s,t,\C}[[x]]=\C$. 

\subsection{The $(s,t)$-antiderivative}

Let $f$ be an arbitrary function. In \cite{nji} the following $(p,q)$-integrals of $f$ in the interval $[0,a]$ are defined 
\begin{align}
    \int_{0}^{a} f(x)d_{p,q}x&=(p-q)a\sum_{k=0}^{\infty}\frac{q^{k}}{p^{k+1}}f\left(\frac{q^{k}}{p^{k+1}}a\right)\text{ if  }\ \Bigg\vert\frac{p}{q}\Bigg\vert>1,\label{eqn_int1}\\
    \int_{0}^{a}f(x)d_{p,q}x&=(q-p)a\sum_{k=0}^{\infty}\frac{p^{k}}{q^{k+1}}f\left(\frac{p^{k}}{q^{k+1}}a\right)\text{ if  }\ \Bigg\vert\frac{p}{q}\Bigg\vert<1\label{eqn_int2}
\end{align}
and the $(p,q)$-integral of $f$ in the interval $[a,b]$ is defined as
\begin{align*}
    \int_{a}^{b}f(x)d_{p,q}x&=\int_{0}^{b}f(x)d_{p,q}x-\int_{0}^{a}f(x)d_{p,q}x.
\end{align*}
Let $I$ denote a set of real numbers. Set $q=\varphi_{s,t}^{\prime}/\varphi_{s,t}$ such that $0<\vert q\vert<1$. We will say that $I$ is a $(q,\varphi_{s,t})$-geometric set if it contains all geometric sequences $\{xq^n/\varphi_{s,t}\}$, $n\in\N_{0}=\N\cup\{0\}$, $x\in I$. For a pair of real numbers $a$ and $b$, we define
\begin{equation*}
    [a,b]_{\varphi}=\Bigg\{a\frac{q^n}{\varphi_{s,t}}\ :\ n\in\N_{0}\Bigg\}\cup\Bigg\{b\frac{q^n}{\varphi_{s,t}}\ :\ n\in\N_{0}\Bigg\}
\end{equation*}
which we will call a $(q,\varphi_{s,t})$-interval with extreme points $a$ and $b$. For real numbers $a$ and $b$, the following property holds:
\begin{equation*}
    \text{If }a,b\in I\ \Longrightarrow\ [a,b]_{\varphi}\subset I.
\end{equation*}
Consider the function $f:[a,b]_{\varphi}\rightarrow\K$, where $\K$ is either $\R$ or $\C$. The $(s,t)$-integral in $[a,b]_{\varphi}$ of $f$ is defined as
\begin{equation}\label{eqn_int_def}
    \int_{a}^{b}f(x)d_{s,t}x=(1-q)\sum_{n=0}^{\infty}\Bigg[bf\left(b\frac{q^n}{\varphi_{s,t}}\right)-af\left(a\frac{Q^n}{\varphi_{s,t}}\right)\Bigg]q^n.
\end{equation}
Equally, set $q$ such that $\vert q\vert>1$. We will say that $I$ is a $(q^{-1},\varphi_{s,t}^{\prime})$-geometric set if it contains all geometric sequences $\{xq^{-n}/\varphi_{s,t}^{\prime}\}$, $n\in\N_{0}$, $x\in I$. Similarly, we define the $(q^{-1},\varphi_{s,t}^{\prime})$-interval with extreme points $a$ and $b$
\begin{equation*}
    [a,b]_{\varphi^{\prime}}=\Bigg\{a\frac{q^{-n}}{\varphi_{s,t}^{\prime}}\ :\ n\in\N_{0}\Bigg\}\cup\Bigg\{b\frac{q^{-n}}{\varphi_{s,t}^{\prime}}\ :\ n\in\N_{0}\Bigg\}.
\end{equation*}
The $(s,t)$-integral in $[a,b]_{\varphi^{\prime}}$ of $f$ is
\begin{equation}\label{eqn_int_def2}
    \int_{a}^{b}f(x)d_{s,t}x=(1-q)\sum_{n=0}^{\infty}\Bigg[bf\left(b\frac{q^{-n}}{\varphi_{s,t}^{\prime}}\right)-af\left(a\frac{q^{-n}}{\varphi_{s,t}^{\prime}}\right)\Bigg]q^{-n}.
\end{equation}
If $[a,b]_{q}$ and $[a,b]_{q^{-1}}$ denote respectively the $q$-interval and $q^{-1}$-interval, 
\begin{align*}
    [a,b]_{q}&=\Big\{aq^{n}\ :\ n\in\N_{0}\Big\}\cup\Big\{bq^{n}\ :\ n\in\N_{0}\Big\},\\
    [a,b]_{q^{-1}}&=\Big\{aq^{-n}\ :\ n\in\N_{0}\Big\}\cup\Big\{bq^{-n}\ :\ n\in\N_{0}\Big\},
\end{align*}
then $[a,b]_{\varphi}=[a/\varphi_{s,t},b/\varphi_{s,t}]_{q}$ shows the relationship between $(q,\varphi_{s,t})$-intervals and $q$-intervals. Likewise, $[a,b]_{\varphi^{\prime}}=[a/\varphi_{s,t}^{\prime},b/\varphi_{s,t}^{\prime}]_{q^{-1}}$. 
From \cite{nji} it is known that if $0<q<1$ and if $\vert f(x)x^{\alpha}\vert$ is bounded on the interval $(0,A]$ for some $0\leq\alpha<1$, then the $(p,q)$-integrals Eqs.(\ref{eqn_int1}) and (\ref{eqn_int2}) convergent to a function $F(x)$ on $(0,A]$, which is a $(p,q)$-antiderivative of $f(x)$, i.e.,
\begin{equation*}
    D_{p,q}\int_{0}^{x}f(r)d_{p,q}r=f(x).
\end{equation*}

In the following result, it is clear that a $(s,t)$-integral under deformation is equal to re-scaling the integration interval.
\begin{proposition}
Set $u\neq0$. For $q\neq1$,
\begin{equation*}
    \int_{a}^{b}f(x)d_{us,u^2t}x=u\int_{a/u}^{b/u}f(x)d_{s,t}x.
\end{equation*}
\end{proposition}
\begin{proof}
By Definition \ref{eqn_int_def},
\begin{align*}
    \int_{a}^{b}f(x)d_{us,u^{2}t}x&=(1-q)\sum_{n=0}^{\infty}\Bigg[bf\left(b\frac{q^n}{u\varphi_{s,t}}\right)-af\left(a\frac{q^n}{u\varphi_{s,t}}\right)\Bigg]q^n\\
    &=u(1-q)\sum_{n=0}^{\infty}\Bigg[\frac{b}{u}f\left(\frac{b}{u}\frac{q^n}{\varphi_{s,t}}\right)-\frac{a}{u}f\left(\frac{a}{u}\frac{q^n}{\varphi_{s,t}}\right)\Bigg]q^n\\
    &=u\int_{a/u}^{b/u}f(x)d_{s,t}x.
\end{align*}
The same result follows by applying Eq.(\ref{eqn_int_def2}).
\end{proof}

\begin{corollary}
\begin{enumerate}
    \item $\int_{a}^{b}f(x)d_{s,t}x=\varphi_{s,t}\int_{a/\varphi_{s,t}}^{b/\varphi_{s,t}}f(x)d_{q}x$
    \item $\int_{a}^{b}f(x)d_{s,t}x=\frac{i}{\sqrt{t}}\int_{-a\sqrt{t}i}^{-b\sqrt{t}i}f(x)d_{U}x$
\end{enumerate}
\end{corollary}
The $(s,t)$-analog of the property of the ordinary integral
\begin{equation*}
    \int_{a}^{b}\alpha f(x)dx=\alpha\int_{a}^{b}f(x)dx,\ \alpha\in\R
\end{equation*}
is given below.
\begin{theorem}\label{theo_integ_cons}
Set $q\neq1$. For all function $G(\log_{q}(x))\in\Per_{s,t}$ and for all function $f$ $(s,t)$-integrable in $[0,r]$, $r\in\{a,b\}$, $0\leq a<b$, and continuous at $0$, it is true that
\begin{equation*}
    \int_{a}^{b}G(\log_{q}(x))f(x)d_{s,t}x=\Bigg[G\left(\log_{q}\left(\frac{r}{\varphi_{s,t}}\right)\right)\int_{0}^{r}f(x)d_{s,t}x\Bigg]_{r=a}^{b}.
\end{equation*}
\end{theorem}
\begin{proof}
Suppose $0<a<b$. By Definition \ref{eqn_int_def} with $0<\vert q\vert<1$
    \begin{align*}
        \int_{a}^{b}G(\log_{q}(x))f(x)d_{s,t}x&=(1-q)\sum_{n=0}^{\infty}\Bigg[bG\left(\log_{q}\left(b\frac{q^n}{\varphi_{s,t}}\right)\right)f\left(b\frac{q^n}{\varphi_{s,t}}\right)\\
        &\hspace{3cm}-aG\left(\log_{q}\left(a\frac{q^n}{\varphi_{s,t}}\right)\right)f\left(a\frac{q^n}{\varphi_{s,t}}\right)\Bigg]q^n.
    \end{align*}
Now, for $q$-periodicity $G(\log_{q}(rq^n/\varphi_{s,t}))=G(\log_{q}(r/\varphi_{s,t}))$, for $r\in\{a,b\}$. Then
\begin{align*}
    \int_{a}^{b}G(\log_{q}(x))f(x)d_{s,t}x&=(1-q)\sum_{n=0}^{\infty}\Bigg[bG\left(\log_{q}\left(\frac{b}{\varphi_{s,t}}\right)\right)f\left(b\frac{q^n}{\varphi_{s,t}}\right)\\
        &\hspace{2cm}-aG\left(\log_{q}\left(\frac{a}{\varphi_{s,t}}\right)\right)f\left(a\frac{q^n}{\varphi_{s,t}}\right)\Bigg]q^n\\
        &=(1-q)G\left(\log_{q}\left(\frac{b}{\varphi_{s,t}}\right)\right)\sum_{n=0}^{\infty}bf\left(b\frac{q^n}{\varphi_{s,t}}\right)q^n\\
        &\hspace{2cm}-(1-q)G\left(\log_{q}\left(\frac{a}{\varphi_{s,t}}\right)\right)\sum_{n=0}^{\infty}af\left(a\frac{q^n}{\varphi_{s,t}}\right)q^n\\
        &=G\left(\log_{q}\left(\frac{b}{\varphi_{s,t}}\right)\right)\int_{0}^{b}f(x)d_{s,t}x\\
        &\hspace{2cm}-G\left(\log_{q}\left(\frac{a}{\varphi_{s,t}}\right)\right)\int_{0}^{a}f(x)d_{s,t}x.
\end{align*}
If $a=0$, then
\begin{align*}
    \int_{0}^{b}G(\log_{q}(x))f(x)d_{s,t}x&=(1-q)b\sum_{n=0}^{\infty}G\left(\log_{q}\left(b\frac{q^n}{\varphi_{s,t}}\right)\right)f\left(b\frac{q^n}{\varphi_{s,t}}\right)q^n\\
    &=(1-q)bG\left(\log_{q}\left(\frac{b}{\varphi_{s,t}}\right)\right)\sum_{n=0}^{\infty}f\left(b\frac{q^n}{\varphi_{s,t}}\right)q^n\\
    &=G\left(\log_{q}\left(\frac{b}{\varphi_{s,t}}\right)\right)\int_{0}^{b}f(x)d_{s,t}x.
\end{align*}
\end{proof}

The following result is of fundamental importance for the theory that will be developed below. For $q$-integral we can see \cite{cardoso}.
\begin{proposition}\label{prop_ineq1}
Let $f$ and $g$ be two $(s,t)$-integrable functions in $[a,b]$ such that $f(x)\leq g(x)$ for all $x\in[a,b]_{\varphi}$. If $a\leq0\leq b$, then
\begin{equation}\label{eqn_ineq1}
    \int_{a}^{b}f(x)d_{s,t}x\leq\int_{a}^{b}g(x)d_{s,t}x.
\end{equation}
\end{proposition}
\begin{proof}
By hypothesis the inequalities 
\begin{equation*}
    f(bq^{n}/\varphi_{s,t})\leq g(bq^{n}/\varphi_{s,t})
\end{equation*}
and 
\begin{equation*}
    f(aq^{n}/\varphi_{s,t})\leq g(aq^{n}/\varphi_{s,t})
\end{equation*}
hold for all $n=0,1,2,.\ldots$. Therefore, since $a\leq0\leq b$, we obtain
\begin{equation*}
    bf(bq^{n}/\varphi_{s,t})-af(aq^{n}/\varphi_{s,t})\leq bg(bq^{n}/\varphi_{s,t})-ag(aq^{n}/\varphi_{s,t}),
\end{equation*}
$n=0,1,2,\ldots$. Thus Eq.(\ref{eqn_ineq1}) follows from Eq.(\ref{eqn_int_def}).
\end{proof}

\begin{corollary}
If $a\leq0\leq b$ and $\vert f\vert$ is $(s,t)$-integrable in $[a,b]$, then $f$ is $(s,t)$-integrable in $[a,b]$ and the following inequality holds:
\begin{equation*}    \Bigg\vert\int_{a}^{b}fd_{s,t}\Bigg\vert\leq\int_{a}^{b}\vert f\vert d_{s,t}.
\end{equation*}
\end{corollary}
\begin{proof}
Follows from Proposition \ref{prop_ineq1} by taking into account that the inequalities $-\vert f(x)\vert\leq f(x)\leq\vert f(x)\vert$ hold for all $x\in[a,b]_{\varphi}$.
\end{proof}

\begin{definition}
For any real number $p\geq1$ and $a,b\in I$, we will denote by $\mathcal{L}_{s,t}^{p}[a,b]$ the set of functions $f:I\rightarrow\K$ such that $\vert f\vert^{p}$ is $(s,t)$-integrable in $[a,b]$, i.e.,
\begin{equation*}
    \mathcal{L}_{s,t}^{p}[a,b]=\Bigg\{f:I\rightarrow\K\ :\ \int_{a}^{b}\vert f\vert^pd_{s,t}<\infty\Bigg\}.
\end{equation*}
We also set
\begin{equation*}
    \mathcal{L}_{s,t}^{\infty}[a,b]=\Bigg\{f:I\rightarrow\K\ :\ \sup_{n\in\N_{0}}\Bigg\{\Bigg\vert f\left(\frac{aq^n}{\varphi_{s,t}}\right)\Bigg\vert,\Bigg\vert f\left(\frac{bq^n}{\varphi_{s,t}}\right)\Bigg\vert\Bigg\}<\infty\Bigg\}.
\end{equation*}
\end{definition}

\begin{theorem}[{\bf The fundamental theorem of $(s,t)$-calculus}]\label{theo_funda}
Let $0<\vert q\vert<1$, and $I=[a,b]$ a $(q,\varphi)$-interval. Fix $a,b\in I$, and let $f:I\rightarrow\K$ be a function such that $\mathbf{D}_{s,t}f\in\mathcal{L}_{s,t}^{1}[a,b]$. Then:
\begin{enumerate}
    \item The equality
    \begin{equation*}
        \int_{a}^{b}(\mathbf{D}_{s,t}f)(x)d_{s,t}x=\Bigg[f(r)-\lim_{n\rightarrow\infty}f\left(rq^n\right)\Bigg]_{r=a}^{b}
    \end{equation*}
    holds, provided the involved limits exist.
    \item In addition, if $f$ is continuous at $0$, then
    \begin{equation*}
        \int_{a}^{b}(\mathbf{D}_{s,t}f)(x)d_{s,t}x=f(b)-f(a).
    \end{equation*}
\end{enumerate}
\end{theorem}
\begin{proof}
Since the equality
\begin{equation*}
    r(\mathbf{D}_{s,t}f)\left(r\frac{q^n}{\varphi_{s,t}}\right)=\frac{f\left(rq^n\right)-f(rq^{n+1})}{(1-q)q^n}
\end{equation*}
holds for $r\in\{a,b\}$, we deduce
\begin{align*}
    \int_{a}^{b}(\mathbf{D}_{s,t}f)(x)d_{s,t}x&=(1-q)\sum_{n=0}^{\infty}\Bigg[b(\mathbf{D}_{s,t}f)\left(b\frac{q^n}{\varphi_{s,t}}\right)-a(\mathbf{D}_{s,t}f)\left(a\frac{q^n}{\varphi_{s,t}}\right)\Bigg]\\
    &=\sum_{n=0}^{\infty}\left((f(bq^n)-f(aq^n))-(f(bq^{n+1})-f(aq^{n+1}))\right)\\
    &=\Bigg[f(r)-\lim_{n\rightarrow\infty}f\left(rq^n\right)\Bigg]_{r=a}^{b}
\end{align*}
where we have used the telescoping property and the hypothesis that the limits $\lim_{n\rightarrow\infty}f(rq^n)$ exist for $r\in\{a,b\}$.
\end{proof}

\begin{theorem}[{\bf The $(s,t)$-integration by parts formula}]\label{theo_partes}
Assume $0<\vert q\vert<1$, and let $I=[a,b]$ be a $(q,\varphi)$-interval. Let $f:I\rightarrow\K$ and $g:I\rightarrow\K$. For $a,b\in I$, the equality
\begin{multline}
    \int_{I}(\mathbf{D}_{s,t}f)(x)g(\varphi_{s,t}^{\prime}x)d_{s,t}x=\Big[(f\cdot g)(s)-\lim_{n\rightarrow\infty}(f\cdot g)(sq^n)\Big]_{s=a}^{b}\\
    -\int_{I}f(\varphi_{s,t}x)(\mathbf{D}_{s,t}g)(x)d_{s,t}x
\end{multline}
holds, provided $f,g\in\mathcal{L}_{s,t}^{1}[a,b]$, $(\mathbf{D}_{s,t}f)(x)$ and $(\mathbf{D}_{s,t}g)(x)$ are bounded for all $x\in[a,b]_{\varphi}$, and the limits exists. If, in addition, $f$ and $g$ are continuous at $0$, then
\begin{equation*}
    \int_{I}(\mathbf{D}_{s,t}f)(x)g(\varphi_{s,t}^{\prime}x)d_{s,t}x=\Big[f\cdot g\Big]_{a}^{b}
    -\int_{I}f(\varphi_{s,t}x)(\mathbf{D}_{s,t}g)(x)d_{s,t}x.
\end{equation*}
\end{theorem}
\begin{proof}
By the $(s,t)$-product rules, Eqs.(\ref{eqn_der_prod1}) and (\ref{eqn_der_prod2}) one has
\begin{equation*}
    g(\varphi_{s,t}^{\prime}x)(\mathbf{D}_{s,t}f)(x)=\mathbf{D}_{s,t}(f(x)g(x))-f(\varphi_{s,t}x)(\mathbf{D}_{s,t}g)(x),
\end{equation*}
hence, integrating both sides of this equality over the interval $[a,b]$ and taking into account Theorem \ref{theo_funda}, the result follows.
\end{proof}

Analogous results to the previous ones are obtained when $\vert q\vert>1$.

\subsection{Deformed $(s,t)$-exponential function}
In this section, the definition of deformed $(s,t)$-exponential function along with its analytical properties are given.
\begin{definition}
Set $s\neq0$. For all $u\in\C$, we define the deformed $(s,t)$-exponential function in $\ward_{s,t,\C}[[z]]$ as
\begin{equation*}
    \exp_{s,t}(z,u)=
    \begin{cases}
        \sum_{n=0}^{\infty}u^{\binom{n}{2}}\frac{z^{n}}{\brk[c]{n}_{s,t}!}&\text{ if }u\neq0;\\
        1+z&\text{ if }u=0.
    \end{cases}
\end{equation*}
Also, we define 
\begin{align*}
\exp_{s,t}(z)&=\exp_{s,t}(z,1),\\
\Exp_{s,t}(z)&=\exp_{s,t}(z,\varphi_{s,t}),\\    
\Exp^{\prime}_{s,t}(z)&=\exp_{s,t}(z,\varphi^{\prime}_{s,t}).
\end{align*}
\end{definition}

\begin{definition}
Set $s,t\in\R/\{0\}$. We define the deformed $q$-exponential and the deformed $U$-exponential functions, respectively, as
\begin{align*}
    \exp_{q}(z,u)&=\sum_{n=0}^{\infty}u^{\binom{n}{2}}\frac{z^n}{\brk[s]{n}_{q}!},\\
    \exp_{2\tau,-1}(z,u)&=1+\sum_{n=1}^{\infty}u^{\binom{n}{2}}\frac{\sin^{n}(\theta_{s,t})}{\sin(\theta_{s,t})\sin(2\theta_{s,t})\cdots\sin(n\theta_{s,t})}z^{n}
\end{align*}
with $\tau=is/2\sqrt{t}$ and $\theta_{s,t}=\arccos(\tau)\neq k\pi$, $k\in\Z$. When $s\neq0$ and $t=0$, 
\begin{equation*}
    \exp_{s,0}(z,u)=\sum_{n=0}^{\infty}\left(\frac{u}{s}\right)^{\binom{n}{2}}z^{n}.
\end{equation*}
If $u=s$
\begin{equation*}
    \Exp_{s,0}(z)=\frac{1}{1-z}\text{ and }\Exp_{s,0}^{\prime}(z)=1+z.
\end{equation*}
\end{definition}
From the above definition, it is very easy to establish that
\begin{align*}
    \Exp_{s,t}(z)&=\exp_{q}(z),\\
    \Exp_{s,t}^{\prime}(z)&=\exp_{q}(z,q),\\
    \exp_{s,t}(z,u)&=\exp_{q}(z,u/\varphi_{s,t}).
\end{align*}
where $q=\varphi_{s,t}^{\prime}/\varphi_{s,t}$.

The following Theorem follows from Theorem \ref{theo_conv_st}.
\begin{theorem}\label{theo_exp_conv}
Set $s\neq0,t\neq0$. If $\vert q\vert<1$, the function $\exp_{s,t}(z,u)$ is
\begin{enumerate}
    \item[1.] an entire function if $\vert u\vert<\vert\varphi_{s,t}\vert$,
    \item[2.] convergent in the disk $\vert z\vert<\vert\varphi_{s,t}\vert/\sqrt{s^{2}+4t}$ when $\vert u\vert=\vert\varphi_{s,t}\vert$,
    \item[3.] convergent in $z=0$ when $\vert u\vert>\vert\varphi_{s,t}\vert$.
\end{enumerate}
Suppose that $\vert q\vert>1$. The function $\exp_{s,t}(z,u)$ is
\begin{enumerate}
    \item[4.] is entire when $\vert u\vert<\vert\varphi_{s,t}^\prime\vert$,
    \item[5.] convergent in the disk $\vert z\vert<\vert\varphi_{s,t}^{\prime}\vert/\sqrt{s^2+4t}$ when $\vert u\vert=\vert\varphi_{s,t}^\prime\vert$,
    \item[6.] convergent in $z=0$ when $\vert u\vert>\vert\varphi_{s,t}^\prime\vert$.
\end{enumerate}
\end{theorem}
It is very straightforward to obtain the following result.
\begin{theorem}
For $s\neq0,t=0$ the function $\exp_{s,0}(az,u)$, $a\neq0$, is
\begin{enumerate}
    \item an entire function if $\vert u\vert<\vert s\vert$,
    \item convergent in the disk $\vert z\vert<1/\vert a\vert$ when $\vert u\vert=\vert s\vert$,
    \item convergent in $z=0$ when $\vert u\vert>\vert s\vert$.
\end{enumerate}
\end{theorem}

\begin{proposition}\label{prop_exp_fibo_neg}
For all $u,v\in\C$, $v\neq0$, and $t\neq0$ the deformed $(s,t)$-exponential functions holds
\begin{equation*}
    \exp_{vs,v^2t}(z,u)=\exp_{s,t}(z,u/v).
\end{equation*} 
If $v=u$, then $\exp_{us,u^2t}(z,u)=\exp_{s,t}(z)$. 
\end{proposition}
\begin{proof}
The proof is as follows
\begin{align*}
    \exp_{vs,v^2t}(z,u)&=\sum_{n=0}^{\infty}u^{\binom{n}{2}}\frac{z^n}{\brk[c]{n}_{vs,v^2t}!}=\sum_{n=0}^{\infty}u^{\binom{n}{2}}\frac{z^n}{v^{\binom{n}{2}}\brk[c]{n}_{s,t}!}=\exp_{s,t}(z,u/v).
\end{align*}
If we make $v=u$, then $\exp_{us,u^2t}(z,u)=\exp_{s,t}(z)$.
\end{proof}

\begin{theorem}\label{theo_diff_expu}
The deformed derivative of the deformed $(s,t)$-exponential function is
\begin{equation*}
    \mathbf{D}_{sv,tv^2}(\exp_{s,t}(z,u))=\exp_{s,t}(uvz,u).
\end{equation*}
\end{theorem}
\begin{proof}
We have that
\begin{align*}
    \mathbf{D}_{sv,tv^2}(\exp_{s,t}(z,u))&=\mathbf{D}_{sv,tv^2}\left(\sum_{n=0}^{\infty}u^{\binom{n}{2}}\frac{z^{n}}{v^{-\binom{n}{2}}\brk[c]{n}_{sv,tv^2}!}\right)\\
    &=\sum_{n=0}^{\infty}u^{\binom{n}{2}}\frac{\mathbf{D}_{sv,tv^2}\left(z^{n}\right)}{v^{-\binom{n}{2}}\brk[c]{n}_{sv,tv^2}!}\\
    &=\sum_{n=1}^{\infty}u^{\binom{n}{2}}\frac{\brk[c]{n}_{sv,tv^2}z^{n-1}}{v^{-\binom{n}{2}}\brk[c]{n}_{sv,tv^2}!}\\
    &=\sum_{n=1}^{\infty}u^{\binom{n}{2}}\frac{z^{n-1}}{v^{-\binom{n}{2}}\brk[c]{n-1}_{sv,tv^2}!}\\
    &=\sum_{n=0}^{\infty}u^{\binom{n}{2}}u^{n}\frac{z^{n}}{v^{-\binom{n}{2}}v^{-n}\brk[c]{n}_{sv,tv^2}!}=\exp_{s,t}(uvz,u)
\end{align*}
and the proof is achieved.
\end{proof}

\begin{corollary}
From Eq.(\ref{eqn_qdiff}), the $q$-derivative of the deformed $(s,t)$-exponential function is
\begin{align*}
    D_q(\exp_{s,t}(z,u))&=\exp_{s,t}(uz/\varphi_{s,t},u).
\end{align*}
In particular
\begin{align*}
D_q(\exp_{s,t}(z))&=\exp_{s,t}(z/\varphi_{s,t}),\\
D_q(\exp_{s,t}^{\prime}(z))&=\exp_{s,t}^\prime(-tz/\varphi_{s,t}),\\
D_q(\Exp_{s,t}(z))&=\Exp_{s,t}(z),\\
D_q(\Exp_{s,t}^{\prime}(z))&=\Exp^{\prime}_{s,t}(\varphi^{\prime}_{s,t}z/\varphi_{s,t}).
\end{align*}
\end{corollary}

\begin{corollary}
From Eq.(\ref{eqn_Udiff}), the $U$-derivative of the deformed $(s,t)$-exponential functions is
\begin{align*}
    D_{U}(\exp_{s,t}(z,u))&=\exp_{s,t}(iuz/\sqrt{t},u).
\end{align*}
In particular
\begin{align*}
D_{U}(\exp_{s,t}(z))&=\exp_{s,t}(iz/\sqrt{t}),\\
D_{U}(\exp_{s,t}^{\prime}(z))&=\exp_{s,t}(-i\sqrt{t}z),\\
D_{U}(\Exp_{s,t}(z))&=\Exp_{s,t}(i\varphi_{s,t}z/\sqrt{t}),\\
D_{U}(\Exp_{s,t}^{\prime}(z))&=\Exp^{\prime}_{s,t}(i\varphi^{\prime}_{s,t}z/\sqrt{t}).
\end{align*}
\end{corollary}
In the following results, we will give the infinite product representation of the functions $\Exp_{s,t}(z)$ and $\Exp_{s,t}^{\prime}(z)$.
\begin{theorem}\label{theo_exp1}
Set $s\neq0$ and $t\neq0$. If $\vert q\vert<1$, then the meromorphic continuation of $\Exp_{s,t}(z)$ is given by
    \begin{equation*}
        \Exp_{s,t}(z)=\prod_{k=0}^{\infty}\frac{\varphi_{s,t}^{k+1}}{\varphi_{s,t}^{k+1}-(\varphi_{s,t}-\varphi_{s,t}^{\prime})\varphi_{s,t}^{\prime k}z}.
    \end{equation*}
If $\vert q\vert>1$, then $\Exp_{s,t}(z)$ is given by
\begin{equation*}
    \Exp_{s,t}(z)=\prod_{k=0}^{\infty}\left(1-(\varphi_{s,t}-\varphi_{s,t}^{\prime})\frac{\varphi_{s,t}^{k}}{\varphi_{s,t}^{\prime(k+1)}}z\right).
\end{equation*}
\end{theorem}
\begin{proof}
As $\mathbf{D}_{s,t}\Exp_{s,t}(z)=\Exp_{s,t}(\varphi_{s,t}z)$, then
\begin{equation*}
    \frac{\Exp_{s,t}(\varphi_{s,t}z)-\Exp_{s,t}(\varphi_{s,t}^{\prime}z)}{(\varphi_{s,t}-\varphi_{s,t}^{\prime})z}=\Exp_{s,t}(\varphi_{s,t}z).
\end{equation*}
Suppose $\vert q\vert<1$. Then
\begin{align*}
    \Exp_{s,t}(\varphi_{s,t}z)&=[1-(\varphi_{s,t}-\varphi_{s,t}^{\prime})z]^{-1}\Exp_{s,t}(\varphi_{s,t}^{\prime}z)\nonumber
\end{align*}
and by setting $w=\varphi_{s,t}z$, 
\begin{equation}
    \Exp_{s,t}(w)=(1-(1-q)w)^{-1}\Exp_{s,t}(qw)\label{eqn_qm1}.
\end{equation}
Now, suppose $\vert q\vert>1$. Then
\begin{equation*}
    \Exp_{s,t}(\varphi_{s,t}^{\prime}z)=[1-(\varphi_{s,t}-\varphi_{s,t}^{\prime})z]\Exp_{s,t}(\varphi_{s,t}z).
\end{equation*}
By setting $w=\varphi_{s,t}^{\prime}z$,
\begin{equation}\label{eqn_qm2}
    \Exp_{s,t}(w)=[1-(q^{-1}-1)w]\Exp_{s,t}(q^{-1}w).
\end{equation}
Iterating Eqs. (\ref{eqn_qm1}) and (\ref{eqn_qm2}) yields desired results.
\end{proof}
Thus the meromorphic function $\prod_{k=0}^{\infty}\frac{1}{1-(1-q)q^kz}$, with simple poles at $\frac{1}{(1-q)q^k}$, $k\in\N_{0}$, is a good substitute for $\Exp_{s,t}(z)$.
The following are the specializations for the infinite product representation of the $(s,t)$-exponential function:
\begin{enumerate}
\item 
\begin{equation*}
     \Exp_{1,1}(z)=\prod_{k=0}^{\infty}\frac{(1+\sqrt{5})^{k+1}}{(1+\sqrt{5})^{k+1}-\sqrt{5}(1-\sqrt{5})^{k}z}.
\end{equation*}
\item
\begin{equation*}
   \Exp_{2,1}(z)=\prod_{k=0}^{\infty}\frac{(1+\sqrt{2})^{k+1}}{(1+\sqrt{2})^{k+1}-\sqrt{2}(1-\sqrt{2})^{k}z}.
\end{equation*}
\item 
\begin{equation*}
   \Exp_{1,2}(z)=\prod_{k=0}^{\infty}\frac{2^{k+1}}{2^{k+1}-3(-1)^{k}z}.
\end{equation*}
\item 
\begin{equation*}
   \Exp_{3,-2}(z)=\prod_{k=0}^{\infty}\frac{1}{1-(z/2^{k+1})}.
\end{equation*}
\item 
\begin{equation*}
     \Exp_{2t,-1}(z)=\prod_{k=0}^{\infty}\frac{(t+\sqrt{t^2-1})^{k+1}}{(t+\sqrt{t^2-1})^{k+1}-\sqrt{t^2-1}(t-\sqrt{t^2-1})^{k}z}.
\end{equation*}
\item 
\begin{equation*}
     \e_{p,q}(z)=\prod_{k=0}^{\infty}\frac{p^{k+1}}{p^{k+1}-(p-q)q^{k}z}.
\end{equation*}
\end{enumerate}
When $s\neq0$ and $t=0$, then
\begin{equation*}
    \Exp_{s,0}(z)=\frac{1}{1-z}.
\end{equation*}
Analogously, the following theorem is proved.

\begin{theorem}\label{theo_exp2}
Set $s\neq0$ and $t\neq0$. If $\vert q\vert<1$, then
    \begin{equation*}
    \Exp_{s,t}^{\prime}(z)=\prod_{k=0}^{\infty}\left(1+(\varphi_{s,t}-\varphi_{s,t}^{\prime})\frac{\varphi_{s,t}^{\prime k}}{\varphi_{s,t}^{k+1}}z\right).
\end{equation*}
If $\vert q\vert>1$, then the meromorphic continuation of $\Exp_{s,t}^{\prime}(z)$ is given by
\begin{equation*}
        \Exp_{s,t}^{\prime}(z)=\prod_{k=0}^{\infty}\frac{\varphi_{s,t}^{\prime(k+1)}}{\varphi_{s,t}^{\prime(k+1)}+(\varphi_{s,t}-\varphi_{s,t}^{\prime})\varphi_{s,t}^{k}z}.
    \end{equation*}
\end{theorem}

From Theorem \ref{theo_exp1} and \ref{theo_exp2} we have the following result.
\begin{corollary}\label{cor_inver_exp}
Set $s\neq0$. For all $q$ with $q\neq1$,
\begin{equation*}
    \Exp_{s,t}(z)\Exp_{s,t}^{\prime}(-z)=\Exp_{s,t}(-z)\Exp_{s,t}^{\prime}(z)=1.
\end{equation*}
\end{corollary}
We end this section by giving inequalities for the function $\exp_{s,t}(x,u)$.
\begin{theorem}\label{theo_ineq_exp}
Set $s,t\in\R$ such that $s\neq0$ and $s^2+4t>0$. The function $\exp_{s,t}(x,u)$ satisfies the following inequalities:
\begin{enumerate}
    \item For $s,t$ fixed and for $0<u<v$, $a\geq0$, it holds
    \begin{equation*}
        1\leq\exp_{s,t}(ax,u)\leq\exp_{s,t}(ax,v),\ 0\leq x<\infty.
    \end{equation*}
    \item Fix $t\in\R$ and tome $s_{1},s_{2}\in\R$ such that  $s_{1}<s_{2}$. Then 
    \begin{equation*}
    \exp_{s_{1},t}(x,u)>\exp_{s_{2},t}(x,u).    
    \end{equation*}
    \item Fix $s\in\R$ and tome $t_{1},t_{2}\in\R$ such that  $t_{1}<t_{2}$. Then
    \begin{equation*}
        \exp_{s,t_{1}}(x,u)>\exp_{s,t_{2}}(x,u).
    \end{equation*}
\end{enumerate}
\end{theorem}
\begin{proof}
Since $0<u<v$ and $a\geq0$, we have for $n\geq0$
\begin{align*}
    u^{\binom{n}{2}}&\leq v^{\binom{n}{2}}\\
    u^{\binom{n}{2}}a^{n}\frac{x^n}{\brk[c]{n}_{s,t}!}&\leq v^{\binom{n}{2}}a^n\frac{x^n}{\brk[c]{n}_{s,t}!}.
\end{align*}
Taking summation over $n$, we get for $0\leq x<\infty$ the desired result. Now set $t\in\R$ and assume $s_{1}<s_{2}$ for $s_{1},s_{2}\in\R$. Then $\brk[c]{2}_{s_{1},t}=s_{1}<s_{2}=\brk[c]{2}_{s_{2},t}$. Assume by induction hypothesis that $\brk[c]{n}_{s_{1},t}<\brk[c]{n}_{s_{2},t}$. In this way
\begin{align*}
    \brk[c]{n+1}_{s_{1},t}&=s_{1}\brk[c]{n}_{s_{1},t}+t\brk[c]{n-1}_{s_{1},t}\\
    &<s_{2}\brk[c]{n}_{s_{2},t}+t\brk[c]{n-1}_{s_{2},t}=\brk[c]{n+1}_{s_{2},t}
\end{align*}
and $\brk[c]{n}_{s_{1},t}<\brk[c]{n}_{s_{2},t}$ for all $n\geq2$. It follows that $\brk[c]{n}_{s_{1},t}!\leq\brk[c]{n}_{s_{2},t}!$ and that $\exp_{s_{1},t}(x,u)>\exp_{s_{2},t}(x,u)$. The last statement is proved in a similar way.
\end{proof}

\begin{theorem}\label{theo_ineq_exp2}
Set $s\neq0$, $s^2+4t>0$.
\begin{enumerate}
    \item If $\vert q\vert<1$ and $x\neq0$, then 
    \begin{equation}\label{eqn_ineq_exp1}
    \exp_{s,t}(x,u)<\Exp_{s,t}^{\prime}(x)<e^{x},    
    \end{equation}
    provided that $u<\varphi_{s,t}^\prime$. If $x<\frac{1}{1-q}$, then
    \begin{equation*}
        e^{x}<\Exp_{s,t}(x).
    \end{equation*}
    \item If $\vert q\vert>1$ and $x\neq0$, then 
    \begin{equation*}
        \exp_{s,t}(x,u)<\Exp_{s,t}(x)<e^{-x},
    \end{equation*}
    provided that $u<\varphi_{s,t}$. If $x<\frac{q}{q-1}$, then
    \begin{equation*}
        e^{x}<\Exp_{s,t}^{\prime}(x).
    \end{equation*}
\end{enumerate}
\end{theorem}
\begin{proof}
Set $\vert q\vert<1$. Denote
\begin{align*}
    P_{n}(x)=\prod_{k=0}^{n}(1+x(1-q)q^k).
\end{align*}
Then we have
\begin{equation*}
    P_{n}(x)<\exp\left(\sum_{k=0}^{n}x(1-q)q^{k}\right)=\exp\left(x(1-q^{n+1})\right),
\end{equation*}
hence
\begin{equation*}
    \Exp_{s,t}^{\prime}(x)=\lim_{n\rightarrow\infty}P_{n}(x)<e^{x},\ \ x>0.
\end{equation*}
On the other side, if $x<\frac{1}{1-q}$, 
\begin{align*}
    \Exp_{s,t}(x)&=\lim_{n\rightarrow\infty}\prod_{k=0}^{n}\frac{1}{1-(1-q)q^kz}\\
    &>\lim_{n\rightarrow\infty}\exp\left(\sum_{k=0}^{n}(1-q)q^{k}x\right)\\
    &=\exp\left((1-q)x\frac{1}{1-q}\right)=e^{x}.
\end{align*}
The remainder of Eq.(\ref{eqn_ineq_exp1}) follows from Theorem \ref{theo_ineq_exp}. The proof for $\vert q\vert>1$ is similar to the previous one.
\end{proof}
The following result follows from Theorem \ref{theo_ineq_exp2}.
\begin{corollary}
Set $s\neq0$ and $s^2+4t>0$. Then
\begin{enumerate}
    \item $\lim_{x\rightarrow\infty}\exp_{s,t}(-x,u)=0$, provided that $0<u<\varphi_{s,t}^{\prime}$ and $\vert q\vert<1$.
    \item $\lim_{x\rightarrow\infty}\Exp_{s,t}(x)=\infty$, provided that $\vert q\vert<1$.
    \item $\lim_{x\rightarrow\infty}\Exp_{s,t}^\prime(x)=\infty$, provided that $\vert q\vert>1$.
    \item $\lim_{x\rightarrow\infty}\Exp_{s,t}(x)=\lim_{x\rightarrow\infty}\exp_{s,t}(x,u)=0$, provided that $\vert q\vert>1$ and $0<u<\varphi_{s,t}$.
\end{enumerate}
\end{corollary}

\section{The equation $\mathbf{D}_{s,t}y(x)=f(x,y(x),y(ux))$}

\subsection{Existence theorem for functional-difference equations with proportional delay}

\begin{definition}
The equation
\begin{equation}\label{eqn_dde}
    (\mathbf{D}_{s,t}y)(x)=f(x,y(x),y(ux)),\ y(0)=y_{0},
\end{equation}
where $\vert f(x,y(x),y(ux))\vert x^{\alpha}$, $0\leq\alpha<1$, is bounded on some rectangle either
\begin{equation*}
    R_{\varphi}=\{\vert x\vert\leq a,\vert y(x)-y_{0}\vert\leq b/\vert\varphi_{s,t}\vert,\vert y(ux)-y_{0}\vert\leq b,a>0,\ b>0\},\hspace{0.5cm}0<u\leq\vert\varphi_{s,t}\vert
\end{equation*}
or
\begin{equation*}
    R_{\varphi\prime}=\{\vert x\vert\leq a,\vert y(x)-y_{0}\vert\leq b/\vert\varphi_{s,t}^\prime\vert,\vert y(ux)-y_{0}\vert\leq b,a>0,\ b>0\},\hspace{0.5cm}0<u\leq\vert\varphi_{s,t}^{\prime}\vert,
\end{equation*}
will be called a functional-difference equation with proportional delay.
\end{definition}

In the following results, our main objective is to prove the existence of solutions to the problem with initial value Eq.(\ref{eqn_dde}).

\begin{theorem}
A function $\phi$ is a solution of Eq.(\ref{eqn_dde}) on an interval $I$ if and only if it is a solution of the $(s,t)$-integral equation 
\begin{equation}\label{eqn_stie}
    y(x)=y_{0}+\int_{0}^{x}f(r,y(r),y(ur))d_{s,t}r
\end{equation}
on $I$.
\end{theorem}
\begin{proof}
Let $\phi$ be a solution of the Eq.(\ref{eqn_dde}) on an interval $I$. Then
\begin{equation}\label{eqn_sol_dde}
    \mathbf{D}_{s,t}\phi(x)=f(x,\phi(x),\phi(ux)),\ \phi(0)=y_{0},
\end{equation}
where either $0<u\leq\vert\varphi_{s,t}\vert$ or $0<u\leq\vert\varphi_{s,t}^{\prime}\vert$. From Theorem \ref{theo_funda}, the equivalent $(s,t)$-integral equation to Eq.(\ref{eqn_sol_dde}) is
\begin{equation}\label{eqn_stie2}
    \phi(x)=\phi(0)+\int_{0}^{x}f(r,\phi(r),\phi(ur))d_{s,t}r
\end{equation}
with $\phi(0)=y_{0}$. Thus $\phi$ is a solution of Eq.(\ref{eqn_stie}). Conversely, suppose that Eq.(\ref{eqn_stie2}) hold. By $(s,t)$-differentiate Eq.(\ref{eqn_stie2}), we get
\begin{equation*}
    \mathbf{D}_{s,t}\phi(x)=f(x,\phi(x),\phi(u x))
\end{equation*}
where either $0<u<\vert\varphi_{s,t}\vert$ or $0<u<\vert\varphi_{s,t}^{\prime}\vert$, for all $x\in I$. From Eq.(\ref{eqn_stie}), $\phi(0)=y_{0}$. Hence $\phi$ is a solution of Eq.(\ref{eqn_dde}).
\end{proof}

\begin{theorem}
Let $\vert f(x,y(x),y(ux))x^{\alpha}\vert$ is bounded on $R_{\varphi}$ for some $0\leq\alpha<1$ and $\vert f\vert\leq M$ on $R_{\varphi}$. The successive approximation 
\begin{equation}\label{eqn_succ_app}
    \phi_{0}(x)=y_{0},\hspace{0.5cm}\phi_{k+1}(x)=y_{0}+\int_{0}^{x}f(r,\phi_{k}(r),\phi_{k}(ur))d_{s,t}r,\hspace{0.5cm}k=0,1,2,3,\ldots,
\end{equation}
$0<u\leq\vert\varphi_{s,t}\vert$, exist on the interval $I_{\varphi}=[-\frac{\alpha}{\vert\varphi_{s,t}\vert},\frac{\alpha}{\vert\varphi_{s,t}\vert}]$, where $\alpha=\min\{a,\frac{b}{M}\}$. If $x\in I_{\varphi}$, then $(x,\phi_{k}(x),\phi_{k}(ux))\in R_{\varphi}$ and $\vert\phi_{k}(x)-y_{0}\vert\leq M\vert x\vert$, $\vert\phi_{k}(ux)-y_{0}\vert\leq Mu\vert x\vert$.
\end{theorem}
\begin{proof}
We prove the result by mathematical induction. Clearly $\phi_{0}(x)=y_{0}$ is $(s,t)$-integrable on $I_{\varphi}$. Thus, the theorem is true for $k=0$. For $k=1$, we have
\begin{align*}
    \phi_{1}(x)&=y_{0}+\int_{0}^{x}f(r,\phi_{0}(r),\phi_{0}(ur))d_{s,t}r\\
    &=y_{0}+\int_{0}^{x}f(r,y_{0},y_{0})d_{s,t}r
\end{align*}
and since $\vert f(x,\phi(x),\phi(ux))x^{\alpha}\vert$ is bounded on $R_{\varphi}$, then $\phi_{1}(x)$ exist. Also
\begin{align*}
    \vert\phi_{1}(x)-y_{0}\vert&=\Bigg\vert\int_{0}^{x}f(r,\phi_{0}(r),\phi_{0}(ur))d_{s,t}r\Bigg\vert\\
    &\leq\int_{0}^{x}\vert f(r,y_{0},y_{0})\vert d_{s,t}r\\
    &\leq M\vert x\vert\\
    &\leq \frac{b}{\vert\varphi_{s,t}\vert},\hspace{0.5cm}x\in I_{\varphi}
\end{align*}
and
\begin{align*}
    \vert\phi_{1}(ux)-y_{0}\vert&=\Bigg\vert\int_{0}^{ux}f(r,\phi_{0}(r),\phi_{0}(ur))d_{s,t}r\Bigg\vert\\
    &\leq\int_{0}^{ux}\vert f(r,y_{0},y_{0})\vert d_{s,t}r\\
    &\leq Mu\vert x\vert,\hspace{0.5cm}0<u\leq\vert\varphi_{s,t}\vert \\
    &\leq b,\ x\in I_{\varphi}.
\end{align*}
Thus, for $x\in I_{\varphi}$, $(x,\phi_{1}(x),\phi_{1}(ux))\in R_{\varphi}$ and $\vert\phi_{1}(x)-y_{0}\vert\leq M\vert x\vert$, $\vert\phi_{1}(ux)-y_{0}\vert\leq Mu\vert x\vert$. The theorem is true for $k=1$. Assume that theorem is true for $k=n$, i.e., for $x\in I_{\varphi}$, $(x,\phi_{n}(x),\phi_{n}(ux))\in R$ and $\vert\phi_{n}(x)-y_{0}\vert\leq M\vert x\vert$, $\vert\phi_{n}(ux)-y_{0}\vert\leq Mu\vert x\vert$. For $k=n+1$
\begin{equation*}
    \phi_{n+1}(x)=y_{0}+\int_{0}^{x}f(r,\phi_{n}(r),\phi_{n}(ur))d_{s,t}r.
\end{equation*}
Since $\vert f(x,\phi_{n}(x),\phi_{n}(ux))x^\alpha\vert$ is bounded on $R_{\varphi}$, $0\leq\alpha<1$, then $\phi_{n+1}(x)$ exists on $I_{\varphi}$. Also,
\begin{align*}
    \vert\phi_{n+1}(x)-y_{0}\vert&\leq M\vert x\vert\\
    &\leq \frac{b}{\vert\varphi_{s,t}\vert},\hspace{0.5cm}x\in I_{\varphi}
\end{align*}
and
\begin{align*}
    \vert\phi_{n+1}(ux)-y_{0}\vert&\leq Mu\vert x\vert\\
    &\leq M\vert\varphi_{s,t}\vert\vert x\vert,\hspace{0.5cm}0<u\leq\vert\varphi_{s,t}\vert\\
    &\leq b,\hspace{0.5cm}x\in I_{\varphi}.
\end{align*}
Thus, if $x\in I_{\varphi}$, $(x,\phi_{n+1}(x),\phi_{n+1}(ux))\in R_{\varphi}$ and $\vert\phi_{n+1}(x)-y_{0}\vert\leq M\vert x\vert$, $\vert\phi_{n+1}(ux)\vert\leq Mu\vert x\vert$.
\end{proof}

\begin{theorem}
Let $\vert f(x,y(x),y(ux))x^{\alpha}\vert$ is bounded on $R_{\varphi^\prime}$ for some $0\leq\alpha<1$ and $\vert f\vert\leq M$ on $R_{\varphi^\prime}$. The successive approximation 
\begin{equation*}
    \phi_{0}(x)=y_{0},\hspace{0.5cm}\phi_{k+1}(x)=y_{0}+\int_{0}^{x}f(r,\phi_{k}(r),\phi_{k}(ur))d_{s,t}r,\hspace{0.5cm}k=0,1,2,3,\ldots,
\end{equation*}
$0<u\leq\vert\varphi_{s,t}^\prime\vert$, exist on the interval $I_{\varphi^\prime}=[-\frac{\alpha}{\vert\varphi_{s,t}^\prime\vert},\frac{\alpha}{\vert\varphi_{s,t}^\prime\vert}]$, where $\alpha=\min\{a,\frac{b}{M}\}$. If $x\in I_{\varphi^\prime}$, then $(x,\phi_{k}(x),\phi_{k}(ux))\in R_{\varphi^\prime}$ and $\vert\phi_{k}(x)-y_{0}\vert\leq M\vert x\vert$, $\vert\phi_{k}(ux)-y_{0}\vert\leq Mu\vert x\vert$.
\end{theorem}

\begin{theorem}[{\bf Existence Theorem 1}]\label{theo_exist}
Let $\vert f(x,y(x),y(ux))x^{\alpha}\vert$ be a function bounded on $R_{\varphi}$, $0\leq\alpha<1$, with $\vert f\vert\leq M$ on $R_{\varphi}$. Further, suppose $f$ satisfies the Lipschitz condition in the second and third variables with Lipschitz constant $L_{1}$ and $L_{2}$ such that
\begin{equation*}
    \vert f(x,y_{1}(x),y_{1}(ux))-f(x,y_{2}(x),y_{2}(ux))\vert\leq L_{1}\vert y_{1}(x)-y_{2}(x)\vert+L_{2}\vert y_{1}(ux)-y_{2}(ux)\vert.
\end{equation*}
Then the successive approximations Eq.(\ref{eqn_succ_app}) converge on the interval $I_{\varphi}=[-\frac{\alpha}{\vert\varphi_{s,t}\vert},\frac{\alpha}{\vert\varphi_{s,t}\vert}]$, where $\alpha=\min\{a,b/M\}$, to a solution $\phi$ of the initial value problem
\begin{equation}\label{eqn_fdu}
    \mathbf{D}_{s,t}y=f(x,y(x),y(ux)),\hspace{0.5cm}y(0)=y_{0}
\end{equation}
on $I_{\varphi}$ when $0<u\leq\vert\varphi_{s,t}\vert$ and $(q,\varphi_{s,t})$ belongs to one of the sets $S_{1},S_{2},S_{3},S_{6},S_{7},S_{8},S_{9}$ of the Theorem \ref{theo_conv_panto} and $I_{\varphi}$ is contained in one the intervals
\begin{align*}
&\left(-\frac{1}{L_{1}\vert1-q\vert},\frac{1}{L_{1}\vert1-q\vert}\right),\ \left(-\frac{1}{L_{2}\vert1-q\vert},\frac{1}{L_{2}\vert1-q\vert}\right),\\
&\left(-\frac{1}{(L_{1}+L_{2}))\vert1-q\vert},\frac{1}{(L_{1}+L_{2}))\vert1-q\vert}\right).
\end{align*}
\end{theorem}
\begin{proof}
Let's note that $\phi_{n}$ can be written as
\begin{equation*}
    \phi_{n}=\phi_{0}+(\phi_{1}-\phi_{0})+(\phi_{2}-\phi_{1})+\cdots+(\phi_{n}-\phi_{n-1}),
\end{equation*}
and hence $\phi_{n}(x)$ is a partial sum for the series
\begin{equation}\label{eqn_series_sol}
    \phi_{0}+\sum_{n=1}^{\infty}[\phi_{n}(x)-\phi_{n-1}(x)].
\end{equation}
Therefore to show that the sequence $\{\phi_{n}(x)\}$ converges is equivalent to showing that the series of the Eq.(\ref{eqn_series_sol}) converges. To prove the latter we must estimate the terms $\phi_{n}(x)-\phi_{n-1}(x)$ of this series. We have that
\begin{align}\label{eqn_phi_1-0}
    \vert\phi_{1}(x)-\phi_{0}(x)\vert&=\Bigg\vert\int_{0}^{x}f(r,y_{0},y_{0})d_{s,t}r\Bigg\vert
    \leq\Bigg\vert\int_{0}^{x}\vert f(r,y_{0},y_{0})\vert d_{s,t}r\Bigg\vert\leq M\vert x\vert.
\end{align}
Writing down the relations defining $\phi_{2}$ and $\phi_{1}$, and subtracting, we obtain
\begin{equation*}
    \phi_{2}(x)-\phi_{1}(x)=\int_{0}^{x}[f(r,\phi_{1}(r),\phi_{1}(ux))-f(r,\phi_{0}(r),\phi_{0}(ur))]d_{s,t}r.
\end{equation*}
Therefore
\begin{align*}
    \vert\phi_{2}(x)-\phi_{1}(x)\vert&\leq\Bigg\vert\int_{0}^{x}\vert f(r,\phi_{1}(r),\phi_{1}(ux))-f(r,\phi_{0}(r),\phi_{0}(ur))\vert d_{s,t}r\Bigg\vert\\
    &\leq\Bigg\vert\int_{0}^{x}[L_{1}\vert\phi_{1}(r)-\phi_{0}(r)\vert+L_{2}\vert\phi_{1}(ur)-\phi_{0}(ur)\vert]d_{s,t}r\Bigg\vert\\
    &\leq M(L_{1}+L_{2}u)\Bigg\vert\frac{x^2}{\brk[c]{2}_{s,t}}\Bigg\vert\\
    &\leq M(L_{1}+L_{2}u)\frac{\vert x\vert^2}{\brk[c]{2}_{\vert s\vert,t}}.
\end{align*}
It can also be shown that
\begin{equation*}
    \vert\phi_{3}(x)-\phi_{2}(x)\vert\leq M(L_{1}+L_{2}u)(L_{1}+L_{2}u^2)\frac{\vert x\vert^3}{\brk[c]{3}_{\vert s\vert,t}!}
\end{equation*}
and
\begin{equation*}
    \vert\phi_{4}(x)-\phi_{3}(x)\vert\leq M(L_{1}+L_{2}u)(L_{1}+L_{2}u^2)(L_{1}+L_{2}u^3)\frac{\vert x\vert^4}{\brk[c]{4}_{\vert s\vert,t}!}.
\end{equation*}
In general, it is easy to show by induction for $n\geq1$ that
\begin{align}\label{eqn_estimate}
    \vert\phi_{n}(x)-\phi_{n-1}(x)\vert&\leq M\prod_{k=1}^{n-1}(L_{1}+L_{2}u^{k})\frac{\vert x\vert^n}{\brk[c]{n}_{\vert s\vert,t}!}=\frac{M}{L_{1}+L_{2}}(L_{1}\oplus L_{2})_{1,u}^{n}\frac{\vert x\vert^n}{\brk[c]{n}_{\vert s\vert,t}!}
\end{align}
where we have used Proposition \ref{prop_abs_nst}. Then
\begin{align*}
    \vert\phi_{n}(x)-y_{0}\vert&=\Bigg\vert\sum_{k=1}^{n}[\phi_{k}(x)-\phi_{k-1}(x)]\Bigg\vert\leq\sum_{k=1}^{n}\vert\phi_{k}(x)-\phi_{k-1}(x)\vert\\
    &\leq\frac{M}{L_{1}+L_{2}}\sum_{k=1}^{n}(L_{1}\oplus L_{2})_{1,u}^{k}\frac{\vert x\vert^k}{\brk[c]{k}_{\vert s\vert,t}!}\\
    &\leq\frac{M}{L_{1}+L_{2}}\sum_{k=1}^{\infty}(L_{1}\oplus L_{2})_{1,u}^{k}\frac{\vert x\vert^k}{\brk[c]{k}_{\vert s\vert,t}!}\\
    &\leq\frac{M}{L_{1}+L_{2}}(\EE_{\vert s\vert,t}(L_{1},L_{2},u;\vert x\vert)-1).
\end{align*}
From Theorem \ref{theo_conv_panto}, the power series for $\EE_{\vert s\vert,t}(L_{1},L_{2},u;\vert x\vert)$ is convergent, then the series
\begin{equation*}
    \vert y_{0}\vert+\sum_{n=1}^{\infty}\vert\phi_{n}(x)-\phi_{n-1}(x)\vert
\end{equation*}
is convergent on $I_{\varphi}$. This implies that the series (\ref{eqn_series_sol}) is convergent on $I_{\varphi}$. Therefore the $n$-th partial sum of Eq.(\ref{eqn_series_sol}), which is just $\phi_{n}(x)$, tends to a limits $\phi(x)$ as $n\rightarrow\infty$, for each $x\in I_{\varphi}$. 

This limit function $\phi$ is a solution to our problem on $I_{\varphi}$. First, let us show that $\phi$ is continuous on $I_{\varphi}$. This may be seen in the following way. If $x_{1},x_{2}\in I$
\begin{align*}
    \vert\phi_{n+1}(x_{1})-\phi_{n+1}(x_{2})\vert=\Bigg\vert\int_{x_{1}}^{x_{2}}f(r,\phi_{n}(r),\phi_{n}(ur))d_{s,t}r\Bigg\vert\leq M\vert x_{1}-x_{2}\vert,
\end{align*}
which implies, by letting $n\rightarrow\infty$,
\begin{equation}\label{eqn_contin}
    \vert\phi(x_{1})-\phi(x_{2})\vert\leq M\vert x_{1}-x_{2}\vert.
\end{equation}
Then $\phi(x)$ is continuous on $I_{\varphi}$. Also, by setting $x_{1}=x$, $x_{2}=0$ in Eq.(\ref{eqn_contin}) we obtain
\begin{align*}
    \vert\phi(x)-y_{0}\vert&\leq M\vert x\vert\hspace{0.5cm}\text{and}\hspace{0.5cm}\vert\phi(ux)-y_{0}\vert\leq Mu\vert x\vert
\end{align*}
which implies that the points $(x,\phi(x),\phi(ux))$ are in $R_{\varphi}$ for all $x\in I_{\varphi}$. We now estimate $\vert\phi(x)-\phi_{p}(x)\vert$. We have
\begin{equation*}
    \phi(x)=\phi_{0}(x)+\sum_{k=1}^{\infty}[\phi_{k}(x)-\phi_{k-1}(x)],
\end{equation*}
and
\begin{equation*}
    \phi_{p}(x)=\phi_{0}(x)+\sum_{k=1}^{p}[\phi_{k}(x)-\phi_{k-1}(x)].
\end{equation*}
Therefore, using Eq.(\ref{eqn_estimate}), we find that 
\begin{align}
    \vert\phi(x)-\phi_{p}(x)\vert&=\Bigg\vert\sum_{n=p+1}^{\infty}[\phi_{n}(x)-\phi_{n-1}(x)]\Bigg\vert\nonumber\\
    &\leq\sum_{n=p+1}^{\infty}\vert\phi_{n}(x)-\phi_{n-1}(x)\vert\nonumber\\
    &\leq\frac{M}{L_{1}+L_{2}}\sum_{n=p+1}^{\infty}(L_{1}\oplus L_{2})_{1,u}^{n}\frac{a^n}{\brk[c]{n}_{\vert s\vert,t}!}\nonumber\\
    &=\frac{M}{L_{1}+L_{2}}\sum_{n=0}^{\infty}(L_{1}\oplus L_{2})_{1,u}^{n+p+1}\frac{a^{n+p+1}}{\brk[c]{n+p+1}_{\vert s\vert,t}!}\nonumber\\
    &\leq\frac{Ma^{p+1}(L_{1}\oplus L_{2})_{1,u}^{p+1}}{(L_{1}+L_{2})\brk[c]{p+1}_{\vert s\vert,t}!}\sum_{n=0}^{\infty}(L_{1}\oplus L_{2}u^{p+1})_{1,u}^n\frac{a^n}{\brk[c]{n}_{\vert s\vert,t}!}\nonumber\\  &=\frac{Ma^{p+1}(L_{1}\oplus L_{2})_{1,u}^{p+1}}{(L_{1}+L_{2})\brk[c]{p+1}_{\vert s\vert,t}!}\EE_{\vert s\vert,t}(L_{1},L_{2}u^{p+1},u;a)\label{eqn_estim2}
\end{align}
Letting $\sigma_{p}=a^{p+1}(L_{1}\oplus L_{2})_{1,u}^{p+1}/\brk[c]{p+1}_{\vert s\vert,t}!$, we see that $\sigma_{p}\rightarrow0$ as $p\rightarrow\infty$, since $\sigma_{p}$ is a general term for the series for $\EE_{\vert s\vert,t}(L_{1},L_{2},u;a)$. In terms of $\sigma_{p}$ Eq.(\ref{eqn_estim2}) may be written as
\begin{equation}\label{eqn_estim1}
    \vert\phi(x)-\phi_{p}(x)\vert\leq \frac{M}{L_{1}+L_{2}}\sigma_{p}\EE_{\vert s\vert,t}(L_{1},L_{2}u^{p+1},u;a).
\end{equation}
To complete the proof we must show that 
\begin{equation}\label{eqn_sol_phi}
    \phi(x)=y_{0}+\int_{0}^{x}f(r,\phi(r),\phi(ur))d_{s,t}r
\end{equation}
for all $x\in I_{\varphi}$. For this purpose, we will show that
\begin{equation}\label{eqn_sol_final}
    \int_{0}^{x}f(r,\phi_{k}(r),\phi_{k}(ur))d_{s,t}r\rightarrow\int_{0}^{x}f(r,\phi(r),\phi(ur))d_{s,t}r
\end{equation}
for all $x$ in $R_{\varphi}$ as $k\rightarrow\infty$. We have
\begin{align}
    &\Bigg\vert\int_{0}^{x}f(r,\phi(r),\phi(ur))d_{s,t}r-\int_{0}^{x}f(r,\phi_{k}(r),\phi_{k}(ur))d_{s,t}r\Bigg\vert\nonumber\\
    &\hspace{3cm}\leq\Bigg\vert\int_{0}^{x}\vert f(r,\phi(r),\phi(ur))-f(r,\phi_{k}(r),\phi_{k}(ur))\vert d_{s,t}r\Bigg\vert\nonumber\\
    &\hspace{3cm}\leq L_{1}\Bigg\vert\int_{0}^{x}\vert\phi(r)-\phi_{k}(r)\vert d_{s,t}r\Bigg\vert+ L_{2}\Bigg\vert\int_{0}^{x}\vert\phi(ur)-\phi_{k}(ur)\vert d_{s,t}r\Bigg\vert\label{eqn_final}
\end{align}
using the fact that $f$ satisfies a Lipschitz condition. The estimate Eq.(\ref{eqn_estim1}) can now used in Eq.(\ref{eqn_final}) to obtain
\begin{align*}
    &\Bigg\vert\int_{0}^{x}f(r,\phi(r),\phi(ur))d_{s,t}r-\int_{0}^{x}f(r,\phi_{k}(r),\phi_{k}(ur))d_{s,t}r\Bigg\vert\\
    &\hspace{1cm}\leq\frac{M(L_{1}+L_{2})}{L_{1}+L_{2}}\sigma_{k}\EE_{\vert s\vert,t}(L_{1},L_{2}u^{k+1},u;a)\vert x\vert=M\sigma_{k}\EE_{\vert s\vert,t}(L_{1},L_{2}u^{k+1},u;a)\vert x\vert
\end{align*}
which tends to zero as $k\rightarrow\infty$, for each $x$ in $I_{\varphi}$. This proves Eq.(\ref{eqn_sol_final}), and hence that satisfies Eq.(\ref{eqn_sol_phi}). Thus our proof of theorem is now complete.
\end{proof}

The estimate Eq.(\ref{eqn_estim2}) of how well the $p$-th approximation $\phi_{p}$ approximates the solution $\phi$ is worthy of special attention.

\begin{theorem}
The $p$-th successive approximation $\phi_{p}$ to the solution $\phi$ of the initial value problem of Theorem \ref{theo_exist} satisfies
\begin{align*}
    \vert\phi(x)-\phi_{p}(x)\vert\leq\frac{Ma^{p+1}(L_{1}\oplus L_{2})_{1,u}^{p+1}}{(L_{1}+L_{2})\brk[c]{p+1}_{\vert s\vert,t}!}\EE_{\vert s\vert,t}(L_{1},L_{2}u^{p+1},u;a)
\end{align*}
for all $x$ in $I_{\varphi}$.
\end{theorem}

\begin{theorem}[{\bf Existence Theorem 2}]\label{theo_exist2}
Let $\vert f(x,y(x),y(ux))x^{\alpha}\vert$ be a function bounded on $R_{\varphi^\prime}$, $0\leq\alpha<1$, with $\vert f\vert\leq M$ on $R_{\varphi^\prime}$. Further, suppose $f$ satisfies the Lipschitz condition in the second and third variables with Lipschitz constant $L_{1}$ and $L_{2}$ such that
\begin{equation*}
    \vert f(x,y_{1}(x),y_{1}(ux))-f(x,y_{2}(x),y_{2}(ux))\vert\leq L_{1}\vert y_{1}(x)-y_{2}(x)\vert+L_{2}\vert y_{1}(ux)-y_{2}(ux)\vert.
\end{equation*}
Then the successive approximations Eq.(\ref{eqn_succ_app}) converge on the interval $I_{\varphi^\prime}=[-\frac{\alpha}{\vert\varphi_{s,t}^\prime\vert},\frac{\alpha}{\vert\varphi_{s,t}^\prime\vert}]$, where $\alpha=\min\{a,b/M\}$, to a solution $\phi$ of the initial value problem
\begin{equation}\label{eqn_fdu2}
    \mathbf{D}_{s,t}y=f(x,y(x),y(ux)),\hspace{0.5cm}y(0)=y_{0}
\end{equation}
on $I_{\varphi^\prime}$ when $0<u\leq\vert\varphi_{s,t}^\prime\vert$ and $(q^{-1},\varphi_{s,t}^\prime)$ belongs to one of the sets $T_{1},T_{2},T_{3},T_{6},T_{7},T_{8},T_{9}$ of the Theorem \ref{theo_conv_panto_prime} and $I_{\varphi^\prime}$ is contained in one the intervals
\begin{align*}
&\left(-\frac{1}{L_{1}\vert q^{-1}-1\vert},\frac{1}{L_{1}\vert q^{-1}-1\vert}\right),\ \left(-\frac{1}{L_{2}\vert q^{-1}-1\vert},\frac{1}{L_{2}\vert q^{-1}-1\vert}\right),\\
&\left(-\frac{1}{(L_{1}+L_{2}))\vert q^{-1}-1\vert},\frac{1}{(L_{1}+L_{2}))\vert q^{-1}-1\vert}\right).
\end{align*}
\end{theorem}

\subsection{Solutions of equation $\mathbf{D}_{s,t}y=f(y(ux))$ via Bell polynomials}

The differential equation $y^{\prime}=f(y)$ is studied in \cite{orozco1,orozco2,orozco3,orozco4}. It is a known fact that Bell polynomials are closely related to the derivatives of the compositions of functions \cite{Faa}, \cite{Foissy}, \cite{Riordan_j}. Let $f$ and $g$ be functions with exponential generating functions $\sum a_{n}\frac{x^{n}}{n!}$ and $\sum b_{n}\frac{x^{n}}{n!}$ 
respectively, with $a_{n},b_{n}\in \C$. Then
\begin{equation}\label{eqn_faa}
f(g(x))= f(b_{0}) + \sum_{n=1}^{\infty}\sum_{k=1}^{n}f^{(k)}(b_{0})B_{n,k}(b_{1},...,b_{n-k+1})\frac{x^{n}}{n!}
\end{equation}
where $B_{n,k}(b_{1},...,b_{n-k+1})$ is the $n$-th partial Bell polynomial
\begin{equation*}
B_{n,k}(b_{1},...,b_{n-k+1})=\sum_{\vert p(n)\vert=k}\frac{n!}{j_{1}!j_{2}!\cdots j_{n}!}\left[\frac{b_{1}}{1!}\right]^{j_{1}}\left[\frac{b_{2}}{2!}\right]^{j_{2}}\cdots \left[\frac{b_{n}}{n!}\right]^{j_{n}}
\end{equation*}
and the sum runs over all partitions $p(n)$ of $n$, that is, $n=j_{1}+2j_{2}+\cdots+nj_{n}$, $j_{h}$ denotes the number of parts of size $h$ and $\vert p(n)\vert=j_{1}+j_{2}+\cdots+j_{n}$ is the length 
of the partition $p(n)$. All the above together with the identity
\begin{equation}\label{eqn_hom}
    B_{n,k}(abx_{1},ab^{2}x_{2},...,ab^{n-k+1}x_{n-k+1})=a^{k}b^{n}B_{n,k}(x_{1},x_{2},...,x_{n-k+1})
\end{equation}
allows us to obtain the functional-difference equation with proportional delay $(s,t)$-analog of $y^{\prime}=f(y)$.

\begin{theorem}
Suppose $f$ with representation in power series. Then the series
\begin{equation*}
    y(x)=y_{0}+f\left(y_{0}\right)x+\sum_{n=1}^{\infty}\frac{u^n}{\brk[c]{n+1}_{s,t}}\sum_{k=1}^{n}f^{(k)}\left(y_{0}\right)B_{n,k}\left(y_{1},\ldots,y_{n-k+1}\right)\frac{x^{n+1}}{n!},
\end{equation*}
where the $y_{i}$ are in terms of $y_{0}$, is solution of the 
autonomous equation with proportional delay 
\begin{equation*}
\mathbf{D}_{s,t}y=f(y(ux)),\ y(0)=y_{0}.
\end{equation*}
\end{theorem}

\begin{corollary}
Suppose $f$ with representation in power series. Then the series
\begin{equation*}
    y(x)=y_{0}+f\left(y_{0}\right)x+\sum_{n=1}^{\infty}u^n\sum_{k=1}^{n}f^{(k)}\left(y_{0}\right)B_{n,k}\left(y_{1},\ldots,y_{n-k+1}\right)\frac{x^{n+1}}{(n+1)!},
\end{equation*}
where the $y_{i}$ are in terms of $y_{0}$, is solution of the 
autonomous equation with proportional delay 
\begin{equation*}
y^\prime=f(y(ux)),\ y(0)=y_{0}.
\end{equation*}
\end{corollary}

In the following result, we show that there is an infinite number of solutions to a functional difference equation with proportional delay if $x$ is taken to be positive.
\begin{theorem}
Suppose $f$ with representation in power series. Take $p(x)=G(\log_{q}(x))$, $x>0$, in $\Per_{s,t}$. The series
\begin{multline}\label{eqn_sol_gen}
    y(x)=p(x)+f\left(p\left(\frac{u}{\varphi_{s,t}}x\right)\right)x\\
    +\sum_{n=1}^{\infty}\frac{u^n}{\brk[c]{n+1}_{s,t}}\sum_{k=1}^{n}f^{(k)}\left(p\left(\frac{u}{\varphi_{s,t}}x\right)\right)B_{n,k}\left(p_{1}\left(\frac{u}{\varphi_{s,t}}x\right),\ldots,p_{n-k+1}\left(\frac{u}{\varphi_{s,t}}x\right)\right)\frac{x^{n+1}}{n!},
\end{multline}
where the $p_{i}(x)$ are functions in $p(x)$, is solution of the 
autonomous equation with proportional delay 
\begin{equation}\label{eqn_assd}
\mathbf{D}_{s,t}y=f(y(ux)).    
\end{equation}
This means that there are an infinite number of solutions to Eq.(\ref{eqn_assd}).
\end{theorem}
\begin{proof}
Let $(p_{n}(x))_{n\in\N_{0}}$ be a sequence of functions in $\Per_{s,t}$ and suppose that $y(x)=\sum_{n=0}^{\infty}p_{n}(x)x^n/n!$ is solution of Eq.(\ref{eqn_assd}). Then Eq.(\ref{eqn_assd}) is equivalent to
\begin{multline*}
    \sum_{n=0}^{\infty}p_{n+1}(\varphi_{s,t}x)\brk[c]{n+1}_{s,t}\frac{x^n}{(n+1)!}\\
    =f(p_{0}(ux))+\sum_{n=1}^{\infty}\sum_{k=1}^{n}f^{(k)}(p_{0}(ux))B_{n,k}(up_{1}(ux),\ldots,u^{n-k+1}p_{n-k+1}(ux))\frac{x^n}{n!}
\end{multline*}
and if set $p_{0}(x)=p(x)$, then it follows that
\begin{align}
    p_{1}(x)&=f\left(p\left(\frac{ux}{\varphi_{s,t}}\right)\right),\label{eqn_termino_1}\\
    p_{n+1}(x)&=\frac{n+1}{\brk[c]{n+1}_{s,t}}\sum_{k=1}^{n}f^{(k)}\left(p\left(\frac{ux}{\varphi_{s,t}}\right)\right)\nonumber\\
    &\hspace{1.5cm}\times B_{n,k}\left(up_{1}\left(\frac{ux}{\varphi_{s,t}}\right),\ldots,u^{n-k+1}p_{n-k+1}\left(\frac{ux}{\varphi_{s,t}}\right)\right)\nonumber\\
    &=\frac{(n+1)u^n}{\brk[c]{n+1}_{s,t}}\sum_{k=1}^{n}f^{(k)}\left(p\left(\frac{ux}{\varphi_{s,t}}\right)\right)B_{n,k}\left(p_{1}\left(\frac{ux}{\varphi_{s,t}}\right),\ldots,p_{n-k+1}\left(\frac{ux}{\varphi_{s,t}}\right)\right)\label{eqn_termino_n_1}
\end{align}
where we have used Eq.(\ref{eqn_hom}). Taking $(s,t)$-derivative to Eq.(\ref{eqn_sol_gen}), we have
\begin{align*}
    \mathbf{D}_{s,t}y(x)&=f(p(ux))\\
    &+\sum_{n=1}^{\infty}\frac{u^n}{n!\brk[c]{n+1}_{s,t}}\sum_{k=1}^{n}f^{(k)}\left(p\left(ux\right)\right)B_{n,k}\left(p_{1}\left(ux\right),\ldots,p_{n-k+1}\left(ux\right)\right)\brk[c]{n+1}_{s,t}x^{n}\\
    &=f(p(ux))+\sum_{n=1}^{\infty}\sum_{k=1}^{n}f^{(k)}\left(p\left(ux\right)\right)B_{n,k}\left(p_{1}\left(ux\right),\ldots,p_{n-k+1}\left(ux\right)\right)\frac{(ux^n)}{n!}\\
    &=f(y(u x))
\end{align*}
which completes the proof.
\end{proof}
Some values of $p_{i}(x)$ are
\begin{align*}
p_{1}(x)&=f\left(p\left(\frac{ux}{\varphi_{s,t}}\right)\right),\\
p_{2}(x)&=\frac{u}{\brk[c]{2}_{s,t}}f^{\prime}\left(p\left(\frac{ux}{\varphi_{s,t}}\right)\right)f\left(p\left(\frac{u^2x}{\varphi_{s,t}^2}\right)\right), \\
p_{3}(x)&=\frac{u^3}{\brk[c]{3}_{s,t}!}f^{\prime}\left(p\left(\frac{ux}{\varphi_{s,t}}\right)\right)f^{\prime}\left(p\left(\frac{u^2x}{\varphi_{s,t}^2}\right)\right)f\left(p\left(\frac{u^3x}{\varphi_{s,t}^3}\right)\right)\\
&\hspace{2cm}+\frac{u^2}{\brk[c]{3}_{s,t}}f^{\prime\prime}\left(p\left(\frac{ux}{\varphi_{s,t}}\right)\right)f^{2}\left(p\left(\frac{u^2x}{\varphi_{s,t}^2}\right)\right).
\end{align*}

\subsection{Solution of equation $\mathbf{D}_{s,t}y(x)=ay(ux)$}

\begin{theorem}\label{theo_sol_pantou}
Set $s\neq0$, $t\neq0$. 
\begin{enumerate}
    \item A solution to the initial value problem
\begin{align}
    \mathbf{D}_{s,t}y(x)&=ay(ux),\label{eqn_pantou}\\
    y(0)&=\xi\nonumber
\end{align}
is
\begin{equation*}
    y(x)=\xi\exp_{s,t}(ax,u)=\xi\sum_{n=0}^{\infty}u^{\binom{n}{2}}a^{n}\frac{x^n}{\brk[c]{n}_{s,t}!}.
\end{equation*}
\item For $x>0$, the solutions of Eq.(\ref{eqn_pantou}) are of the form
\begin{equation}\label{eqn_sol_1}
y(x)=\e_{s,t}(a,x,u,p)=\sum_{n=0}^{\infty}u^{\binom{n}{2}}a^{n}p\left(\frac{u^n}{\varphi_{s,t}^{n}}x\right)\frac{x^n}{\brk[c]{n}_{s,t}!}
\end{equation}
where $p(x)=G(\log_{q}(x))$, with $G$ a periodic function with period one.
\end{enumerate}
\end{theorem}
\begin{proof}
From Theorem \ref{theo_funda},
\begin{equation*}
    y(x)=\xi+\int_{0}^{x}y(ux)d_{s,t}x.
\end{equation*}
By successive approximation method define
\begin{align*}
    y_{0}(x)&=\xi,\\
    y_{k+1}(x)&=\xi+\int_{0}^{x}y_{k}(ur)d_{s,t}r.
\end{align*}
Then
\begin{align*}
    y_{0}(x)&=\xi,\\
    y_{1}(x)&=\xi+a\xi x,\\
    y_{2}(x)&=\xi+a\xi x+\xi u\frac{a^2}{\brk[c]{2}_{s,t}}x^{2},\\
    y_{3}(x)&=\xi+a\xi x+\xi u\frac{a^2}{\brk[c]{2}_{s,t}!}x^{2}+\xi u^{3}\frac{a^3}{\brk[c]{3}_{s,t}!}x^{3},\\
    &\vdots\\
    y_{k}(x)&=\xi\sum_{n=0}^{k}a^{n}u^{\binom{n}{2}}\frac{x^{n}}{\brk[c]{n}_{s,t}!}
\end{align*}
and
\begin{equation*}
    y(x)=\lim_{k\rightarrow\infty}y_{k}(x)=\xi\sum_{n=0}^{\infty}u^{\binom{n}{2}}a^{n}\frac{x^n}{\brk[c]{n}_{s,t}!}=\xi\exp_{s,t}(ax,u).
\end{equation*}
From Eqs. (\ref{eqn_termino_1}) and (\ref{eqn_termino_n_1}) with $f(x)=ax$, we have
\begin{align*}
    p_{1}(x)&=ap\left(\frac{ux}{\varphi_{s,t}}\right),\\
    p_{n+1}(x)&=a\frac{(n+1)u^n}{\brk[c]{n+1}_{s,t}}B_{n,1}\left(p_{1}\left(\frac{ux}{\varphi_{s,t}}\right),\ldots,p_{n}\left(\frac{ux}{\varphi_{s,t}}\right)\right)\\
    &=a\frac{(n+1)u^n}{\brk[c]{n+1}_{s,t}}p_{n}\left(\frac{ux}{\varphi_{s,t}}\right).
\end{align*}
Therefore
\begin{equation*}
    p_{n}(x)=\frac{a^{n}n!u^{\binom{n}{2}}}{\brk[c]{n}_{s,t}!}p\left(\frac{u^{n}}{\varphi_{s,t}^{n}}x\right)
\end{equation*}
and
\begin{equation*}
    y(x)=\sum_{n=0}^{\infty}\frac{a^{n}u^{\binom{n}{2}}}{\brk[c]{n}_{s,t}!}p\left(\frac{u^{n}}{\varphi_{s,t}^{n}}x\right)x^{n}
\end{equation*}
is a solution of Eq.(\ref{eqn_pantou}).
\end{proof}
As 
\begin{equation*}
    G\left(\log_{q}\left(\frac{u^n}{\varphi_{s,t}^n}z\right)\right)=G\left(\log_{q}\left(\frac{u^n}{\varphi_{s,t}^{\prime n}}z\right)\right)
\end{equation*}
for all periodic function $G$ with period one and all $n\in\N$, then
\begin{equation*}
\sum_{n=0}^{\infty}u^{\binom{n}{2}}a^{n}G\left(\log_{q}\left(\frac{u^n}{\varphi_{s,t}^{\prime n}}z\right)\right)\frac{z^n}{\brk[c]{n}_{s,t}!}
\end{equation*}
are also solutions of Eq.(\ref{eqn_pantou}). 

The $(p,q)$-power is defined by Sadjang in \cite{nji} by
\begin{equation}\label{eqn_pq_binomial}
(x\ominus a)_{p,q}^{n}=
\begin{cases}
1,& \text{ if }n=0;\\
\prod_{k=0}^{n-1}(p^{k}x-q^{k}a),& \text{ if }n\geq1.
\end{cases}
\end{equation}
In \cite{jagann} the following expansion is obtained
\begin{equation}\label{eqn_bin_expan}
    (x\ominus a)_{p,q}^n=\sum_{k=0}^{n}\fibonomial{n}{k}_{p,q}p^{\binom{k}{2}}q^{\binom{n-k}{2}}x^{k}(-a)^{n-k}.
\end{equation}

\begin{definition}
For all $s\neq0$, $s^2+4t>0$, we define the binomial $(s,t)$-exponential function as
\begin{align*}
\exp_{s,t}(a(x\ominus y)_{\varphi,\varphi^\prime})&=\sum_{n=0}^{\infty}a^n\frac{(x\ominus y)_{\varphi,\varphi^\prime}^{n}}{\brk[c]{n}_{s,t}!}=\Exp_{s,t}(ax)\Exp_{s,t}^{\prime}(-ay).
\end{align*}
\end{definition}

\begin{theorem}
Let $a$ be a complex number and set $s\neq0$, $s^2+4t>0$. Then the initial value problem
\begin{equation}\label{eqn_panto1}
    \mathbf{D}_{s,t}y(x)=ay(\varphi_{s,t}x),
\end{equation}
$y(0)=\xi$, have solution 
\begin{equation*}
f(x)=\xi\Exp_{s,t}(ax).
\end{equation*}
A solution of Eq.(\ref{eqn_panto1}) with initial value $y(\eta)=\xi$ is
\begin{equation*}
    f(x)=\xi\Exp_{s,t}(ax)\Exp_{s,t}^{\prime}(-a\eta)=\xi\exp_{s,t}(a(x\ominus\eta)_{\varphi,\varphi^\prime}).
\end{equation*}
Also, if $f(x)=\exp_{s,t}(a(x\ominus\eta)_{\varphi,\varphi^\prime})$ is a solution of Eq.(\ref{eqn_panto1}) with $f(\eta)=1$, $\eta>0$, then $g(x)=G(\log_{q}(x))f(x)$ is also a solution for every $G(\log_{q}(x))$ in $\Per_{s,t}$ and $g(\eta)=G(\log_{q}(\eta))$. 
\end{theorem}
\begin{proof}
The solution follows from Theorem \ref{theo_sol_pantou} with $u=\varphi_{s,t}$. As
\begin{align*}
    \mathbf{D}_{s,t}(\xi\exp_{s,t}(a(x\ominus\eta)_{\varphi,\varphi^\prime})&=\xi\Exp_{s,t}^{\prime}(-a\eta)\mathbf{D}_{s,t}\Exp_{s,t}(ax)\\
    &=a\xi\Exp_{s,t}^{\prime}(-a\eta)\Exp_{s,t}(a\varphi_{s,t}x)\\
    &=a\xi\exp_{s,t}(a(\varphi_{s,t}x\ominus\eta)_{\varphi,\varphi^\prime}),
\end{align*}
then $f(x)=\xi\exp_{s,t}(a(x\ominus\eta))$ is solution of Eq.(\ref{eqn_panto1}) with initial value $f(\eta)=\xi$. Suppose $f(x)=\exp_{s,t}(a(x\ominus\eta))$ is a solution of Eq.(\ref{eqn_panto1}). Then
\begin{align*}
    \mathbf{D}_{s,t}(G(\log_{q}(x))f(x))&=G(\log_{q}(\varphi_{s,t}x))\mathbf{D}_{s,t}f(x)\\
    &=aG(\log_{q}(\varphi_{s,t}x))f(\varphi_{s,t}x)
\end{align*}
and $G(\log_{q}(x))f(x)$ is also a solution for all $\eta>0$ and for all functions $G(\log_{q}(x))\in\Per_{s,t}$.
\end{proof}

Similarly, we have the following theorem.
\begin{theorem}
Let $a$ be a complex number and set $s\neq0$, $s^2+4t>0$. Then the initial value problem
\begin{equation}\label{eqn_panto2}
    \mathbf{D}_{s,t}y(x)=ay(\varphi_{s,t}^\prime x),
\end{equation}
$y(0)=\xi$, have solution 
\begin{equation*}
f(x)=\xi\Exp_{s,t}^\prime(ax).
\end{equation*}
A solution of Eq.(\ref{eqn_panto1}) with initial value $y(\eta)=\xi$ is
\begin{equation*}
    f(x)=\xi\Exp_{s,t}^\prime(ax)\Exp_{s,t}(-a\eta)=\xi\exp_{s,t}(a((-\eta)\oplus x)_{\varphi,\varphi^\prime}).
\end{equation*}
Also, if $f(x)=\exp_{s,t}(a((-\eta)\oplus x)_{\varphi,\varphi^\prime})$ is a solution of Eq.(\ref{eqn_panto1}) with $f(\eta)=1$, $\eta>0$, then $g(x)=G(\log_{q}(x))f(x)$ is also a solution for every $G(\log_{q}(x))$ in $\Per_{s,t}$ and $g(\eta)=G(\log_{q}(\eta))$. 
\end{theorem}

Next, we will give some properties of the functions $\e_{s,t}(a,z,u,p)$.
\begin{proposition}
Set $s\neq0$ and $s^2+4t>0$. For all $a,u\in\R$ and $p(x)=f(\log_{q}(z)),q(x)=g(\log_{q}(x))\in\Per_{s,t}$. Then
\begin{enumerate}
    \item $\e_{s,t}(a,z,u,0)=0$.
    \item $\e_{s,t}(a,z,u,c)=c\exp_{s,t}(az,u)$.
    \item $\e_{s,t}(a,z,\varphi_{s,t},p)=f(\log_{q}(z))\Exp_{s,t}(az)$.
    \item $\e_{s,t}(a,z,\varphi_{s,t}^{\prime},p)=f(\log_{q}(z))\Exp_{s,t}^{\prime}(az)$.
    \item $\e_{s,t}(a,z,u,p+q)=\e_{s,t}(a,z,u,p)+\e_{s,t}(a,z,u,q)$.
    \item $\e_{s,t}(a,z,q^{-m}u,p)=e_{q^{m}s,q^{2m}t}(a,z,u,p)$, for all $m\in\Z$.
\end{enumerate}
\end{proposition}
\begin{proof}
By direct application of Eq. (\ref{eqn_sol_1}).
\end{proof}
From the previous proposition it follows that the function $\e_{s,t}(a,z,u,p)$ generalizes to the $(s,t)$-exponential functions $\Exp_{s,t}(az,u)$ and $\Exp_{s,t}^{\prime}(az,u)$. In the following theorems, elementary analytic properties of $\e_{s,t}(a,z,u,p)$ are shown.

\begin{theorem}
Take $p(x)=f(\log_{q}(x))\in\Per_{s,t}$ such that $f(x)$ is a continuous periodic function with period one with $\vert f(x)\vert\leq M$ for all $x\in\R$. Then
\begin{equation*}
    \e_{s,t}(a,x,u,p)\leq M\exp_{s,t}(ax,u)
\end{equation*}
\end{theorem}
\begin{proof}
Since $f$ is continuous and periodical, then there exists a number $M>0$ such that $f$ is bounded with $\vert p(x)\vert<M$ for all $z$. Then
\begin{align*}
    \e_{s,t}(a,z,u,p)&=\sum_{n=0}^{\infty}u^{\binom{n}{2}}a^{n}p\left(\frac{u^n}{\varphi_{s,t}^n}x\right)\frac{x^n}{\brk[c]{n}_{s,t}!}\\
    &\leq M\sum_{n=0}^{\infty}u^{\binom{n}{2}}a^{n}\frac{x^n}{\brk[c]{n}_{s,t}!}\\
    &=M\exp_{s,t}(ax,u)
\end{align*}    
and $\exp_{s,t}(ax,u)$ is an upper bound of $\e_{s,t}(a,x,u,p)$. 
\end{proof}

\begin{corollary}
Set $s\neq0$ and $s^2+4t>0$ and take $p(x)=f(\log_{q}(x))\in\Per_{s,t}$ such that $\vert f(x)\vert\leq M$ for all $x>0$. Then
\begin{enumerate}
    \item $\lim_{x\rightarrow\infty}\e_{s,t}(a,x,u,p)=\infty$ provided that $u>0$.
    \item $\lim_{x\rightarrow\infty}\e_{s,t}(a,-x,u,p)=0$, provided that $\vert q\vert>1$ and $u<\varphi_{s,t}$.
    \item $\lim_{x\rightarrow\infty}\e_{s,t}(a,-x,u,p)=0$, provided that $\vert q\vert<1$ and $u<\varphi_{s,t}^\prime$.
\end{enumerate}
\end{corollary}

\begin{example}
Find the solution to the functional equation
\begin{equation}\label{eqn_example}
    f(3x)-f(2x)=xf(x/2),\ f(0)=1.
\end{equation}
We have that $\varphi_{s,t}=3$ and $\varphi_{s,t}^{\prime}=2$, with $s=5$ and $t=-6$. Therefore, the solution is
\begin{equation*}
    f(x)=\exp_{5,-6}(x,1/2)=\sum_{n=0}^{\infty}\left(\frac{1}{2}\right)^{\binom{n}{2}}\frac{x^n}{\brk[c]{n}_{5,-6}!},
\end{equation*}
convergent in $\R$, where
\begin{equation*}
    \brk[c]{n}_{5,-6}=(0,1,5,19,65,211,\ldots).
\end{equation*}
The Eq.(\ref{eqn_example}) with initial value $f(\eta)=\xi$, $\eta>0$ have solutions
\begin{equation*}
    f(x)=\e_{5,-6}(1,x,1/2,p)=\sum_{n=0}^{\infty}\left(\frac{1}{2}\right)^{\binom{n}{2}}G\left(\log_{2/3}\left(\frac{x}{6^n}\right)\right)\frac{x^n}{\brk[c]{n}_{5,-6}!},
\end{equation*}
for each $p(x)=G(\log_{2/3}(x))\in\Per_{5,-6}$. If $G(x)=\sin(2\pi x)$, then 
\begin{equation*}
\e_{5,-6}(1,x,1/2,p)\leq\exp_{5,-6}(x,1/2)    
\end{equation*}
and $\e_{5,-6}(1,x,1/2,p)$ is convergent in $\R$, too. Also, 
\begin{equation*}
    \lim_{x\rightarrow\infty}\e_{5,-6}(1,-x,1/2,p)=0.
\end{equation*}
\end{example}

\subsection{Solution of $\mathbf{D}_{s,t}y=ay(x)+by(ux)$}

\begin{theorem}\label{theo_st_panto}
The Pantograph equation
\begin{equation*}
    \mathbf{D}_{s,t}y=ay(x)+by(ux),\ y(0)=1
\end{equation*}
has solution
\begin{equation*}
    y(x)=\sum_{n=0}^{\infty}(a\oplus b)_{1,u}^{n}\frac{x^n}{\brk[c]{n}_{s,t}!}
\end{equation*}
\end{theorem}
\begin{proof}
By successive approximation method, we obtain
\begin{align*}
    y_{0}(x)&=1,\\
    y_{k+1}(x)&=1+\int_{0}^{x}(ay_{k}(r)+by(ur))d_{s,t}r.
\end{align*}
Then
\begin{align*}
    y_{0}(x)&=1,\\
    y_{1}(x)&=1+(a+b)x,\\
    y_{2}(x)&=1+(a+b) x+(a+b)(a+bu)\frac{x^2}{\brk[c]{2}_{s,t}!},\\
    y_{3}(x)&=1+(a+b)x+(a+b)(a+bu)\frac{x^2}{\brk[c]{2}_{s,t}!}\\
    &\hspace{1cm}+(a+b)(a+bu)(a+bu^2)\frac{x^3}{\brk[c]{3}_{s,t}!}.
\end{align*}
In general
\begin{equation*}
    y_{k}(x)=1+\sum_{n=1}^{k}\prod_{i=0}^{n-1}(a+bu^{i})\frac{x^n}{\brk[c]{n}_{s,t}!}
\end{equation*}
and
\begin{equation*}
    y(x)=\lim_{k\rightarrow\infty}y_{k}(x)=1+\sum_{n=1}^{\infty}\prod_{i=0}^{n-1}(a+bu^{i})\frac{x^n}{\brk[c]{n}_{s,t}!}.
\end{equation*}
\end{proof}

\begin{definition}\label{def_EE}
We define the $(1,u)$-deformed $(s,t)$-exponential function as
    \begin{equation*}
        \EE_{s,t}(a,b,u;z)=\sum_{n=0}^{\infty}(a\oplus b)_{1,u}^n\frac{z^n}{\brk[c]{n}_{s,t}!}.
    \end{equation*}
When $u=q$, we define $\EE_{q}(a,b;z)=\EE_{1+q,-q}(a,b,q;z)$.
\end{definition}
Some special values of the function $\EE_{s,t}(a,b,u;x)$ are
\begin{align*}
    \EE_{s,t}(a,0,-;x)&=\exp_{s,t}(ax),\\
    \EE_{s,t}(0,a,u;x)&=\exp_{s,t}(ax,u),\\
    \EE_{s,t}(a,b,1;x)&=\exp_{s,t}((a+b)x),\\
    \EE_{s,t}(ac,bc,u;x)&=\EE_{s,t}(a,b,u;cx).
\end{align*}

\begin{theorem}\label{theo_conv_panto}
Suppose that $a\neq0$ and $b\neq0$. If $\vert u\vert<\vert\varphi_{s,t}\vert$, then the function $\EE_{s,t}(a,b,u;z)$ is 
\begin{enumerate}
    \item entire only if either $(q,\varphi_{s,t})\in S_{1}$ or $(q,\varphi_{s,t})\in S_{2}$,
    \begin{equation*}
        S_{1}=\{(q,\varphi_{s,t}):\vert q\vert\neq1,\ \vert\varphi_{s,t}\vert>1\},\ S_{2}=\{(q,\varphi_{s,t}):\vert q\vert>1,\ 1\leq\vert\varphi_{s,t}^{-1}\vert<\vert q\vert\}
    \end{equation*}
\item convergent in $\vert z\vert<\frac{1}{\vert1-q\vert\vert a\vert}$ only if $(q,\varphi_{s,t})\in S_{3}$,
\begin{equation*}
   S_{3}=\{(q,\varphi_{s,t}):0<\vert q\vert<1,\ \vert\varphi_{s,t}\vert=1\},
\end{equation*}
\item convergent in $z=0$ if either $(q,\varphi_{s,t})\in S_{4}$ or $(q,\varphi_{s,t})\in S_{5}$,
\begin{equation*}
S_{4}=\{(q,\varphi_{s,t}):0<\vert q\vert<1,\ 0<\vert\varphi_{s,t}\vert<1\},\ S_{5}=\{(q,\varphi_{s,t}):\vert q\vert>1,\ \vert\varphi_{s,t}^{-1}\vert>\vert q\vert\}.    
\end{equation*}
\end{enumerate}
If $\vert u\vert=\vert\varphi_{s,t}\vert$, then the function $\EE_{s,t}(a,b,u;z)$ is
\begin{enumerate}
\item entire only if either $(q,\varphi_{s,t})\in S_{6}$ or $(q,\varphi_{s,t})\in S_{7}$, $a+b=0$,
\begin{equation*}
S_{6}=\{(q,\varphi_{s,t}):\vert q\vert>1,\ \vert\varphi_{s,t}\vert>1\},\ S_{7}=\{(q,\varphi_{s,t}):0<\vert q\vert<1,\ \vert\varphi_{s,t}\vert=1\}, \end{equation*}
\item convergent in $\vert z\vert<\frac{1}{\vert b\vert\vert1-q\vert}$ only if $(q,\varphi_{s,t})\in S_{8}$,
\begin{equation*}
S_{8}=\{(q,\varphi_{s,t}):0<\vert q\vert<1,\ \vert\varphi_{s,t}\vert>1\}, \end{equation*}
\item convergent in $\vert z\vert<\frac{1}{\vert a+b\vert\vert1-q\vert}$ only if $(q,\varphi_{s,t})\in S_{9}$, $a+b\neq0$,
\begin{equation*}
S_{9}=\{(q,\varphi_{s,t}):0<\vert q\vert<1,\ \vert\varphi_{s,t}\vert=1\}.
\end{equation*}
\end{enumerate}
\end{theorem}
\begin{proof}
Define $a_{n}=(a\oplus b)_{1,u}^{n}\frac{z^n}{\brk[c]{n}_{s,t}!}$. Then
\begin{align*}
    \Big\vert\frac{a_{n+1}}{a_{n}}\Big\vert&=\Bigg\vert\frac{(a\oplus b)_{1,u}^{n+1}z^{n+1}}{\brk[c]{n+1}_{s,t}!}\cdot\frac{\brk[c]{n}_{s,t}!}{(a\oplus b)_{1,u}^{n}z^n}\Bigg\vert\\
    &=\vert z\vert\Bigg\vert\frac{a+bu^n}{\brk[c]{n+1}_{s,t}}\Bigg\vert\\
    &=\vert z\vert\vert1-q\vert\Bigg\vert\frac{a}{\varphi_{s,t}^n}+b\left(\frac{u}{\varphi_{s,t}}\right)^n\Bigg\vert\Bigg\vert\frac{1}{1-q^{n+1}}\Bigg\vert.
\end{align*}
Now, apply each condition for $q$ and $\varphi_{s,t}$.
\end{proof}

\begin{theorem}\label{theo_conv_panto_prime}
Suppose that $a\neq0$ and $b\neq0$. If $\vert u\vert<\vert\varphi_{s,t}^\prime\vert$, then the function $\EE_{s,t}(a,b,u;z)$ is 
\begin{enumerate}
    \item entire only if either $(q^{-1},\varphi_{s,t}^{\prime})\in T_{1}$ or $(q^{-1},\varphi_{s,t}^{\prime})\in T_{2}$,
    \begin{align*}
        T_{1}&=\{(q^{-1},\varphi_{s,t}^{\prime}):\vert q^{-1}\vert\neq1,\ \vert\varphi_{s,t}^{\prime}\vert>1\},\\
        T_{2}&=\{(q^{-1},\varphi_{s,t}^{\prime}):\vert q^{-1}\vert>1,\ 1\leq\vert\varphi_{s,t}^{\prime(-1)}\vert<\vert q^{-1}\vert\},
    \end{align*}
\item convergent in $\vert z\vert<\frac{1}{\vert q^{-1}-1\vert\vert a\vert}$ only if $(q^{-1},\varphi_{s,t}^{\prime})\in T_{3}$,
\begin{equation*}
   T_{3}=\{(q^{-1},\varphi_{s,t}^{\prime}):0<\vert q^{-1}\vert<1,\ \vert\varphi_{s,t}^\prime\vert=1\},
\end{equation*}
\item convergent in $z=0$ if either $(q^{-1},\varphi_{s,t}^{\prime})\in T_{4}$ or $(q^{-1},\varphi_{s,t}^{\prime})\in T_{5}$,
\begin{align*}
T_{4}&=\{(q^{-1},\varphi_{s,t}^\prime):0<\vert q^{-1}\vert<1,\ 0<\vert\varphi_{s,t}^\prime\vert<1\},\\
T_{5}&=\{(q^{-1},\varphi_{s,t}^\prime):\vert q^{-1}\vert>1,\ \vert\varphi_{s,t}^{\prime(-1)}\vert>\vert q^{-1}\vert\}.    
\end{align*}
\end{enumerate}
If $\vert u\vert=\vert\varphi_{s,t}^\prime\vert$, then the function $\EE_{s,t}(a,b,u;z)$ is
\begin{enumerate}
\item entire only if either $(q^{-1},\varphi_{s,t}^{\prime})\in T_{6}$ or $(q^{-1},\varphi_{s,t}^{\prime})\in T_{7}$, $a+b=0$,
\begin{align*}
T_{6}&=\{(q^{-1},\varphi_{s,t}^{\prime}):\vert q^{-1}\vert>1,\ \vert\varphi_{s,t}^{\prime}\vert>1\},\\
T_{7}&=\{(q^{-1},\varphi_{s,t}^{\prime}):0<\vert q^{-1}\vert<1,\ \vert\varphi_{s,t}^\prime\vert=1\}, 
\end{align*}
\item convergent in $\vert z\vert<\frac{1}{\vert b\vert\vert q^{-1}-1\vert}$ only if $(q^{-1},\varphi_{s,t}^{\prime})\in T_{8}$,
\begin{equation*}
T_{8}=\{(q^{-1},\varphi_{s,t}^\prime):0<\vert q^{-1}\vert<1,\ \vert\varphi_{s,t}^\prime\vert>1\}, 
\end{equation*}
\item convergent in $\vert z\vert<\frac{1}{\vert a+b\vert\vert q^{-1}-1\vert}$ only if $(q^{-1},\varphi_{s,t}^{\prime})\in T_{9}$, $a+b\neq0$,
\begin{equation*}
T_{9}=\{(q^{-1},\varphi_{s,t}^\prime):0<\vert q^{-1}\vert<1,\ \vert\varphi_{s,t}^\prime\vert=1\}.
\end{equation*}
\end{enumerate}
\end{theorem}

\begin{theorem}\label{theo_prod_EE}
For $0<\vert q\vert<1$
    \begin{equation*}
        \EE_{q}(a,b;z)=\prod_{k=0}^{\infty}\frac{1+b(1-q)q^{k}z}{1-a(1-q)q^{k}z}=\frac{\E_{1/q}(bz)}{\E_{1/q}(-az)}
    \end{equation*}
where $\E_{1/q}(z)=\sum_{n=0}^{\infty}q^{\binom{n}{2}}z^{n}/[n]_{q}!=\prod_{n=0}^{\infty}(1+(1-q)q^nz)$.
For $\vert q\vert>1$
\begin{equation*}
        \EE_{q}(a,b;z)=\prod_{k=0}^{\infty}\frac{1-a(q^{-1}-1)q^{-k}z}{1+b(q^{-1}-1)q^{-k}z}.
    \end{equation*}
\end{theorem}
\begin{proof}
Suppose $0<\vert q\vert<1$ and write $y(x)=\EE_{q}(a,b;x)$. As $D_{q}y(x)=ay(x)+by(qx)$, then
\begin{equation*}
    y(x)-y(qx)=a(1-q)xy(x)+b(1-q)xy(qx)
\end{equation*}
and
\begin{equation}\label{eqn_EE}
    y(x)=y(qx)\frac{1+b(1-q)x}{1-a(1-q)x}.
\end{equation}
Iterating Eq.(\ref{eqn_EE}) we get
\begin{equation*}
    y(x)=y(q^{n+1}x)\prod_{k=0}^{n}\frac{1+b(1-q)q^{k}x}{1-a(1-q)q^{k}x}
\end{equation*}
and therefore
\begin{equation*}
    y(x)=\lim_{n\rightarrow\infty}y(q^{n+1}x)\prod_{k=0}^{n}\frac{1+b(1-q)q^{k}x}{1-a(1-q)q^{k}x}=y(0)\prod_{k=0}^{\infty}\frac{1+b(1-q)q^{k}x}{1-a(1-q)q^{k}x}.
\end{equation*}
The proof for $\vert q\vert>1$ is similar.
\end{proof}

\begin{proposition}
For all $a,b\in\C$
    \begin{enumerate}
        \item $\EE_{q}(a,b;z)\EE_{q}(b,a;-z)=1$.
        \item $\EE_{q^{-1}}(a,b;z)=\EE_{q}(b,a;z)$.
        \item $\EE_{q}(1,-q;z)=\frac{1}{1-(1-q)z}$.
        \item $\EE_{q}(-q,1;z)=1+(1-q)z$.
    \end{enumerate}
\end{proposition}
\begin{proof}
To prove 1. we use the Definition \ref{def_EE} and Theorem \ref{theo_prod_EE}
\begin{align*}
    \EE_{q}(a,b;z)\EE_{q}(b,a;-z)&=\frac{E_{1/q}(bz)}{E_{1/q}(-az)}\cdot\frac{E_{1/q}(-az)}{E_{1/q}(bz)}=1.
\end{align*}
As $(a\oplus b)_{1,q^{-1}}^{n}=q^{-\binom{n}{2}}(b\oplus a)_{1,q}^{n}$, then
\begin{align*}
    \EE_{q^{-1}}(a,b;z)&=\sum_{n=0}^{\infty}(a\oplus b)_{1,q^{-1}}^{n}\frac{z^n}{[n]_{q^{-1}}!}\\
    &=\sum_{n=0}^{\infty}q^{-\binom{n}{2}}(b\oplus a)_{1,q}^{n}\frac{z^n}{q^{-\binom{n}{2}}[n]_{q}!}\\
    &=\EE_{q}(b,a;z)
\end{align*}
and this proves 2. Statements 3 and 4 arise easily by making $a=1$ and $b=-q$ and then applying 1.
\end{proof}

\begin{example}
The solution of the functional equation
\begin{equation}\label{eqn_example2}
    f(3x)-f(2x)=xf(x)+xf(x/2),\ f(0)=1
\end{equation}
is
\begin{equation*}
    f(x)=\E_{5,-6}(1,1,1/2;x)=\sum_{n=0}^{\infty}(1\oplus1)_{1,1/2}^{n}\frac{x^n}{\brk[c]{n}_{5,-6}!},
\end{equation*}
convergent in $\R$, since $(2/3,3)\in S_{1}$.
\end{example}

From Theorem \ref{theo_st_panto} we obtain the solution of the $(s,t)$-analog of the Ambartsumian equation
\begin{equation*}
    y^{\prime}=-y(x)+\frac{1}{v}y\left(\frac{1}{v}x\right),\ v>1,
\end{equation*}
which describes the surface brightness in Astronomy \cite{ambart,patade}. A $q$-analog of this equation can be found in \cite{abdu}. In the following theorem the $(s,t)$-analog of the Ambartsumian equation is shown.
\begin{theorem}
The solution of the $(s,t)$-Ambartsumian equation
\begin{equation*}
    \mathbf{D}_{s,t}y(x)=-y(x)+\frac{1}{v}y\left(\frac{1}{v}x\right),\ v>1,
\end{equation*}
is
\begin{equation*}
    \EE_{s,t}(-1,v^{-1},v^{-1};x)=\xi+\xi\sum_{n=1}^{\infty}\prod_{k=1}^{n}(v^{-k}-1)\frac{x^{n}}{\brk[c]{n}_{s,t}!}.
\end{equation*}
\end{theorem}




\section{Statements and Declarations}
\subsection{Conflict of Interests}
We have no conflict of interest to disclose.

\subsection{Data availability}
The author confirms that no data known is used in the manuscript.

\end{document}